\documentclass[a4paper, 12pt, twoside, reqno]{amsart}

\usepackage[utf8]{inputenc}
\usepackage[T1]{fontenc}

\usepackage{geometry}
\usepackage{amsmath}
\usepackage{amsthm}
\usepackage{amssymb}
\usepackage{amsfonts}
\usepackage{mathtools}
\usepackage[inline]{enumitem}
\usepackage{fancyhdr}
\usepackage{physics}
\usepackage[mathscr]{euscript}
\usepackage{color}
\usepackage{graphicx}

\usepackage{tikz}
\usepackage[most]{tcolorbox}
\usepackage[colorinlistoftodos]{todonotes}
\usetikzlibrary{tikzmark}

\usepackage{xspace}

\allowdisplaybreaks

\numberwithin{equation}{section}

\newtheorem{theorem}{Theorem}[section]
\newtheorem{proposition}[theorem]{Proposition}
\newtheorem{corollary}[theorem]{Corollary}
\newtheorem{lemma}[theorem]{Lemma}

\theoremstyle{definition}
\newtheorem{definition}[theorem]{Definition}

\theoremstyle{remark}

\usepackage{hyperref}
\hypersetup{%
colorlinks=true, %
linkcolor=blue, %
citecolor=blue, %
filecolor=blue, %
urlcolor=blue, %
pdfauthor={Ronaldo B. Assun\c{c}\~{a}o,
Ol\'{\i}mpio H. Miyagaki, and
Rafaella F. S. Siqueira},
pdftitle={Fractional Sobolev-Chochard critical equation with Hardy term and weighted singularities}, %
pdfsubject={Elliptic partial differential equations}, %
pdfkeywords={Fractional p-Laplacian operator, doubly critical singular problem, variational
methods, weighted Sobolev spaces, weighted Morrey spaces, Caffarelli-Kohn-Nirenberg inequality.}, %
pdfproducer={Latex}, %
pdfcreator={ps2pdf}}

\usepackage[hyperpageref]{backref}

\title[Fractional $p$-Kirchhoff with Sobolev and Choquard nonlinearities]{Fractional $p$-Kirchhoff equation with Sobolev and Choquard singular nonlinearities}

\author[Assun\c{c}\~{a}o]{Ronaldo B. Assun\c{c}\~{a}o}
\address{Ronaldo B. Assun\c{c}\~{a}o \hfill\break\indent
Departamento de Matem\'{a}tica\,---\,
Universidade Federal de Minas Gerais, 
UFMG %\hfill\break\indent
}
\email{ronaldo@mat.ufmg.br}

\author[Miyagaki]{Ol\'{\i}mpio H. Miyagaki}
\thanks{Ol\'{\i}mpio H. Miyagaki was supported by Grant 2022/16407-1\,---\,S\~{a}o Paulo Research Foundation (FAPESP) and Grant 303256/2022-2\,---\,CNPq/Brazil.}
\address{Ol\'{\i}mpio H. Miyagaki \hfill\break\indent
Departamento de Matem\'{a}tica\,---\,
Universidade Federal de S\~{a}o Carlos, 
UFSCar %\hfill\break\indent
} 
\email{ohmiyagaki@gmail.com}

\author[Siqueira]{Rafaella F. S. Siqueira}
\thanks{Rafaella F. S. Siqueira was partially supported by CNPq/Brazil.}
\address{Rafaella F. S. Siqueira 
\hfill\break\indent
Departamento de Matem\'{a}tica\,---\,
Centro Federal de Educa\c{c}\~{a}o Tecnol\'{o}gica, 
CEFET %\hfill\break\indent
}
\email{rafaella.siqueira@cefetmg.br}

\date{Belo Horizonte, \today}

\keywords{%
Fractional $p$-Laplacian operator,
doubly critical singular problem,
variational methods, 
weighted Sobolev spaces,
weighted Morrey spaces,
Caffarelli-Kohn-Nirenberg inequality.}

\subjclass[2010]{%
Primary: %
35B33;
35J92; % Quasilinear elliptic equations with $p$-Laplacian
35R11.     
Secondary: %
35A23,
35B38, % Critical points
35J20. % Variational methods for second-order elliptic equations
}

\begin{document}
\begin{abstract}
In the present work, we consider the following fractional $p$-Kirchhoff equation in the entire space $\mathbb{R}^{N}$ featuring doubly nonlinearities, involving a generalized nonlocal Choquard subcritical term together with a local critical Sobolev term; the problem also includes a Hardy-type term; additionaly, all terms have critical singular weights.
More precisely, we deal with the problem
  \begin{multline*}
  %\label{problema:Kirchhof}
  m(\|u\|^p_{W_{\theta}^{s,p}(\mathbb{R}^N)})\Bigl[(-\Delta)^{s}_{p,\theta} u 
   +V(x) \dfrac{|u|^{p-2}u}{|x|^{\alpha}} \Bigr]  \\
   = \dfrac{|u|^{p^*_s(\beta,\theta)-2}u}%
     {|x|^{\beta}} 
     + \lambda\left[
       I_{\mu} \ast \dfrac{F_{\delta,\theta,\mu}(\cdot, u)}{|x|^{\delta}} 
       \right](x)\dfrac{f_{\delta,\theta,\mu}(x,u)}{|x|^{\delta}}
   \end{multline*}
where 
$0 < s  < 1$; $0<\alpha<N-\mu$;
$0 < \beta < sp + \theta < N$; 
$0 < \mu < N$; $2\delta + \mu < N$; $p^*_s(\beta, \theta)= p(N-\beta)/(N-sp-\theta)$. The function $m \colon \mathbb{R}_0^+ \to \mathbb{R}^+$ is a Kirchhoff function; the potential function $V \colon \mathbb{R}^N \to \mathbb{R}^+$ is continuous; the nonlinearity $f \colon \mathbb{R} \to \mathbb{R}$ is continuous and define $F(s)=\int _0^sf(t)\dd t$; the $I_{\mu}\colon \mathbb{R}^N \to \mathbb{R}$ is defined by $I_{\mu}(x)= |x|^{-\mu}$ and is called the Riesz potential. The fractional $p$-Laplacian operator is defined 
for $\theta=\theta_{1}+\theta_{2}$, $x\in\mathbb{R}^{N}$, and any function $u \in C_{0}^{\infty}(\mathbb{R}^{N})$, as
\begin{align*}
(-\Delta)_{p,\theta}^{s}u(x)
&\coloneqq
\textup{p.v.\xspace}
\int_{\mathbb{R}^{N}} 
\dfrac{|u(x)-u(y)|^{p-2}(u(x)-u(y))}{|x|^{\theta_{1}}
|x-y|^{N+sp}|y|^{\theta_{2}}}\dd{y},
\end{align*} 
where p.v.\xspace\ is the
Cauchy's principal value. Our result improve upon previous work in the following ways: we focus our attention on the existence of a nontrivial weak solution for fractional $p$-Kirchhoff equation in the entire space $\mathbb{R}^{N}$. The possibility of a slower growth in the nonlinearity makes it more difficult to establish a compactness condition; to do so, we use the Cerami condition. The crucial points in our argument are the uniform boundedness of the convolution part and the lack of compactness of the Sobolev embeddings.
\end{abstract}

\maketitle
\tableofcontents

\section{Model problem and main result}
In the present work, we consider the following fractional $p$-Kirchhoff equation in the entire space $\mathbb{R}^{N}$ featuring doubly nonlinearities, involving a generalized nonlocal Choquard subcritical term together with a local critical Sobolev term; the problem also includes a Hardy-type term; additionaly, all terms have critical singular weights.
More precisely, we deal with the problem
  \begin{multline}
  \label{problema:Kirchhof}
  m(\|u\|^p_{W_{\theta}^{s,p}(\mathbb{R}^N)})\Bigl[(-\Delta)^{s}_{p,\theta} u 
   +V(x) \dfrac{|u|^{p-2}u}{|x|^{\alpha}} \Bigr]  \\
   = \dfrac{|u|^{p^*_s(\beta,\theta)-2}u}%
     {|x|^{\beta}} 
     + \lambda\left[
       I_{\mu} \ast \dfrac{F_{\delta,\theta,\mu}(\cdot, u)}{|x|^{\delta}} 
       \right](x)\dfrac{f_{\delta,\theta,\mu}(x,u)}{|x|^{\delta}}
   \end{multline}
where 
$0 < s  < 1$; $0<\alpha<N-\mu$;
$0 < \beta < sp + \theta < N$; 
$0 < \mu < N$; $2\delta + \mu < N$; $p^*_s(\beta, \theta)= p(N-\beta)/(N-sp-\theta)$. The function $m \colon \mathbb{R}_0^+ \to \mathbb{R}^+$ is a Kirchhoff function; the potential function $V \colon \mathbb{R}^N \to \mathbb{R}^+$ is continuous; the nonlinearity $f \colon \mathbb{R} \to \mathbb{R}$ is continuous and define $F(s)=\int _0^sf(t)\dd t$; the $I_{\mu}\colon \mathbb{R}^N \to \mathbb{R}$ is defined by $I_{\mu}(x)= |x|^{-\mu}$ and is called the Riesz potential. The fractional $p$-Laplacian operator is defined 
for $\theta=\theta_{1}+\theta_{2}$, $x\in\mathbb{R}^{N}$, and any function $u \in C_{0}^{\infty}(\mathbb{R}^{N})$, as
\begin{align*}
(-\Delta)_{p,\theta}^{s}u(x)
&\coloneqq
\textup{p.v.}
\int_{\mathbb{R}^{N}} 
\dfrac{|u(x)-u(y)|^{p-2}(u(x)-u(y))}{|x|^{\theta_{1}}
|x-y|^{N+sp}|y|^{\theta_{2}}}\dd{y},
\end{align*} 
where p.v. is the
Cauchy's principal value.

Let us now introduce the spaces of functions that are meaningful to our considerations. Throughout this work, 
we denote the norm of the weighted Lebesgue space
$L^p_V(\mathbb{R}^{N},|x|^{-\eta})$ by
\begin{align*}
\|u\|_{L^p_V(\mathbb{R}^{N};|x|^{-\eta})}
& \coloneqq
 \Bigl(
\int_{\mathbb{R}^{N}} \dfrac{V(x)|u|^{p}}{|x|^{\eta}}\dd{x} \Bigr)^{\frac{1}{p}}
\end{align*}
for any 
$0\leqslant \eta < N$ and $1 \leqslant p < +\infty$.

We can equip the homogeneous fractional Sobolev space 
${W}^{s,p}_{V, \theta}(\mathbb{R}^N)$ 
with the norm 
\begin{align*}
\|u \|_{{W}_{V,\theta}^{s,p} (\mathbb{R}^N)} 
 = \|u \|_{W} & \coloneqq
\Bigl(
 \iint_{\mathbb{R}^{2N}} 
\dfrac{|u(x)-u(y)|^{p}}%
{|x|^{\theta _1}|x-y|^{N+sp}|y|^{\theta _2}}
\dd x \dd y
+
\int_{\mathbb{R}^N}  
\dfrac{V(x)|u|^{p}}{|x|^{\alpha}}
\dd x
\Bigr)^{\frac{1}{p}}\\
& \coloneqq \Big(
[u]_{{W}_{\theta}^{s,p} (\mathbb{R}^N)}^{p}
+ \|u \|_{L^p_V (\mathbb{R}^N; |x|^{-\alpha})}^{p}
\Bigr)^{\frac{1}{p}}.
\end{align*} 
The embedding $W^{s,p}_{V, \theta}(\mathbb{R}^N)\hookrightarrow
L^{\nu}_V(\mathbb{R}^N,|x|^{-\alpha})$ is continuous for any $\nu\in [p,\frac{p(N-\beta)}{N-ps-\theta}]$ and $0< \alpha < N-\mu$, namely there exists a positive constant
$C_\nu$ such that
\begin{align}\label{sobem}
\|u\|_{L^{\nu}_V(\mathbb{R}^N,|x|^{-\alpha})}\leqslant C_\nu \|u\|_W\quad\text{for all }
u\in W^{s,p}_{V,\theta}(\mathbb{R}^N).
\end{align}

The potential function  $V\colon \mathbb{R}^N \to \mathbb{R}^+$ verifies the following assumption
\begin{enumerate}[label=\eqref{hip:V}, 
ref=$V$]
    \item \label{hip:V} $V$ is continuous and there exists $V_0>0$ such that $\inf_{\mathbb{R}^N} V \geqslant V_0.$
\end{enumerate}

Moreover, we assume that the nonlinearities $f, F \colon \mathbb{R}^N \times \mathbb{R} \to \mathbb{R}$ verify the hypotheses
\begin{enumerate}[label=($F\sb{\arabic*}$), 
ref=$F\sb{\arabic*}$]
    \item \label{hip:F1} $F \in C^1(\mathbb{R}, \mathbb{R})$.
    \item \label{hip:F2} There exist constants 
    \begin{align}
    \label{def:psus}
       p_s^{\flat}(\delta, \mu) \coloneqq \dfrac{(N-\delta -\mu/2)p}{N} < q_1 \leqslant q_2 < \dfrac{(N-\delta -\mu/2)p}{N-sp-\theta} \eqqcolon p_s^{\sharp}(\delta, \theta, \mu)
    \end{align}
    and  $c_0>0$ such that for all $t \in \mathbb{R}$, $|f(t)|\leqslant c_0 (|t|^{q_1-1}+ |t|^{q_2-1})$.
    The exponents $p_s^{\flat}(\delta,\mu)$ and $p_s^{\sharp}(\delta, \theta, \mu)$ are named lower and upper critical exponents in the sense of Hardy-Littlewood-Sobolev inequality.
    \item \label{hip:F3} $\lim\limits_{|u(x)| \to \infty} \dfrac{F(u(x))}{|x|^{\delta}|u(x)|^{p \xi}}= \infty$ uniformly with respect to $x \in \mathbb{R}^N$ for $\xi \in [1, 2p_s^{\flat}(\delta, \mu)/p)$.
    \item  \label{hip:F4} There exist constants $r_0\geqslant 0$, $\kappa>\dfrac{N-\beta}{ps+\theta-\beta}$ and $c_1\geqslant 0$ such that for $|t|\geqslant r_0$, 
\begin{align*}
\frac{|F(t)|}{|x|^{\delta}}^\kappa\leqslant c_1|t|^{\kappa p}\mathscr{F}(t),\qquad 
\mathscr{F}(t)\coloneqq\dfrac{1}{p\xi}\dfrac{f(t)}{|x|^{\delta}}t-\dfrac{1}{2}\dfrac{F(t)}{|x|^{\delta}}\geqslant 0.
\end{align*}
\end{enumerate}

With respect to the Kirchhoff function  $m \colon \mathbb{R}_0^+ \to \mathbb{R}^+$ we make the following assumptions.
\begin{enumerate}[label=($m\sb{\arabic*}$), 
ref=$m\sb{\arabic*}$]
    \item \label{hip:m1} $m$ is a continuous function and there exists $m_0>0$ such that $\inf_{t \geqslant 0} m(t)=m_0$.
    \item \label{hip:m2} There exists $\xi \in [1, 2p_s^{\flat}(\delta, \mu)/p)$ such that $m(t)t \leqslant \xi M(t)$ for all $t\geqslant 0$,
    where $M(t)= \int_0^t m(\tau)\dd \tau$.
\end{enumerate}
A typical example is $m(t)=a + b \xi t ^{\xi -1}$ for $t\geqslant 0$, 
where $a \geqslant 0$, $b \geqslant 0$, $a+b>0$, $\xi \in (1,2p_s^{\flat}(\delta, \mu)/p)$ if $b>0$ and $\xi = 1$ if $b=0$; this is called non-degenerate when $a>0$ and $b\geqslant 0$ and is called degenerate if $a=0$ and $b>0$.

Our main goal in this work is to show that 
problem~\eqref{problema:Kirchhof} admits at least a nontrivial weak solution, by which term we mean a function 
$u \in W_{V,\theta}^{s,p}(\mathbb{R}^N)$ such that
\begin{multline*}
m (\|u\|^p_W)\\
\times\left[
\iint_{\mathbb{R}^{2N}} 
\dfrac{|u(x)-u(y)|^{p-2}
(u(x)-u(y))(\phi(x)-\phi(y))}{|x|^{\theta_1}|x-y|^{N+sp}|y|^{\theta_2}
}
\dd{x}\dd {y}
 + 
 \int_{\mathbb{R}^{N}} 
 \dfrac{V(x)|u|^{p-2}u \phi}{|x|^{\alpha}}\dd{x}\right] \\
 = \int_{\mathbb{R}^{N}} 
\dfrac{|u|^{p_{s}^{\ast}(\beta,\theta)-2}u \phi}%
{|x|^{\beta}}\dd{x} + \lambda\int_{\mathbb{R}^{N}}
 \Bigl(I_{\mu} \ast \frac{F_{\delta, \theta, \mu}(u)}{|x|^{\delta}}\Bigr)\frac{f_{\delta, \theta, \mu}(u)}{|x|^{\delta}}\phi \dd x
\end{multline*}
for any test function $\phi \in {W}^{s,p}_{V,\theta}(\mathbb{R}^N)$. 

Now we define the energy functional 
$I \colon {W}^{s,p}_{V,\theta}(\mathbb{R}^N)
\to \mathbb{R}$ by
\begin{align}
\label{kirchhoff:funcionalI}
I(u)
& \coloneqq
\dfrac{1}{p}
M(\|u\|^p_{W})
-\dfrac{1}{p_{s}^{\ast}(\beta,\theta)}
\int_{\mathbb{R}^{N}} 
\dfrac{|u|^{p_{s}^{\ast}(\beta,\theta)}}{|x|^{\beta}} \dd{x} 
-\dfrac{\lambda}{2 }
\iint_{\mathbb{R}^{2N}}
\dfrac{F_{\delta, \theta, \mu}(u(x))F_{\delta, \theta, \mu}(u(y))}%
       {|x|^{\delta}
        |x-y|^{\mu}
        |y|^{\delta}}
        \dd{x}\dd{y} \nonumber \\
& \eqqcolon \Phi(u) - \Xi(u) - \lambda \Psi(u).\end{align}
For the parameters in the previously specified intervals, the energy functional $I$ is well defined and is continuously differentiable, 
i.e., \xspace $I \in C^{1}({W}^{s,p}_{V, \theta}(\mathbb{R}^N);\mathbb{R})$; moreover, a nontrivial critical point of the energy functional $I$ is a nontrivial weak solution to problem~\eqref{problema:Kirchhof}.

The main result is as follows.

\begin{theorem}
\label{teo:chen}
Let $0<\mu<ps+ \theta <N$; suppose \eqref{hip:V}, \eqref{hip:m1} -- \eqref{hip:m2} and \eqref{hip:F1} -- \eqref{hip:F4} hold. Then problem~\eqref{problema:Kirchhof} has a nontrivial weak solution for any $\lambda>0$.
\end{theorem}

\section{Historical background}
The study of nonlocal problems driven by the fractional and nonlocal operators has received a tremendous popularity because of their
intriguing structure and the great variety of phenomena occurring in real-world applications that can be modeled by these
equations such as optimization, finance, phase transition phenomena, anomalous diffusion etc. 
For more information on nonlocal and fractional problems, see the excelent survey paper by Di Nezza,
Palatucci \& Valdinoci~\cite{MR2944369}; see also the book by Molica Bisci, R\u{a}dulescu \& Servadei~\cite{MR3445279}.

\subsection{The Choquard equation}
On the Euclidean space $\mathbb{R}^{N}$, the equation
\begin{align*}
-\Delta u + V(x) u & = (I_{\mu} \ast |u|^{q}) |u|^{q-2}u \qquad (x\in \mathbb{R}^{N})
\end{align*}
was introduced by Choquard in the case $N=3$ and $q=2$ to model one-component plasma. It had appeared earlier in the model of the polaron by Fr\"{o}lich and Pekar, where free electrons interact with the polarisation that they create on the medium. A remarkable feature in the Choquard nonlinearity is the appearance of a lower nonlinear restriction, usually called the lower critical exponent $2^{\flat}>1$, that is, the nonlinearity is superlinear. When $V(x) \equiv 1$, the groundstate solutions exist if 
$2^{\flat}\coloneqq 2(N-\mu/2))/N < q < 2(N-\mu/2)/(N-2s) \coloneqq 2^{\sharp}$ due to the mountain
pass lemma or the method of the Nehari manifold, 
while there are no nontrivial solution if $q = 2^{\flat}$ or if $q = 2^{\sharp}$ as
a consequence of the Pohozaev identity.

In general, the associated Schr\"{o}dinger-type evolution equation
$i \partial_{t} \psi 
= \Delta \psi
+ \big( I_{\mu} \ast |\psi|^{2} \big) \psi$
is a model for large systems of atoms with an attractive interaction that is weaker and has a longer range than that of the nonlinear Schr\"{o}dinger equation. 
Standing wave solutions of this equation
are solutions to the Choquard equation. 
For more information on the various results related to the non-fractional
Choquard-type equations and their variants see the survey by Moroz \& Van Schaftingen~\cite{MR3625092}.

\subsection{Kirchhoff type problems}

Kirchhoff type problems have been widely studied in recent years. After Lions has presented an abstract functional framework to use for Kirchhoff type equations, this problem has been widely studied in extensive literature. Alves, Corr\^{e}a \& Figueiredo~\cite{alves2010class} investigated the existence of positive solutions to the class of nonlocal boundary value problems of the Kirchhoff type with the classical Laplace operator. For other papers involving the Kirchhoff type problems with the classical Laplace operator see Chen \& Li~\cite{chen2013multiple} and Figueiredo~\cite{ figueiredo2013existence}. For the $p$-Laplacian case, Colasuanno \& Pucci~\cite{colasuonno2011multiplicity} established the existence of infinitely many solutions for Dirichlet problems involving the $p$-polyharmonic operators on bounded domains. For the fractional Kirchhoff problem, Fiscella \& Valdinoci~\cite{fiscella2014critical} proposed a stationary fractional Kirchhoff variational problem which takes into account the nonlocal aspect of the tension arising
from nonlocal measurements of the fractional length of the string. It was pointed out in Pucci, Xiang \& Zhang~\cite{MR3975603} that equations like 
$-\bigl(a+b\int_{\mathbb{R}^{N}} |\nabla u|^{2}\dd{x}\bigr)\Delta u + u = k f(u)+|u|^{2^{*}-2}u
$
in the whole space $\mathbb{R}^{N}$ can be applied to describe the growth and movement of a specific species.
Song \& Shi~\cite{song2017existence} considered a class of degenerate fractional $p$-Laplacian equation of Schrödinger–Kirchhoﬀ with critical Hardy–Sobolev nonlinearities. The main feature and diﬃculty of this article is the fact that the Kirchhoﬀ term could be zero at zero, that is, the problem equation is degenerate.
The degenerate case was
studied by Autuori, Fiscella \& Pucci in~\cite{autuori2015stationary}, by introducing a new technical approach based on the asymptotic property of the critical mountain pass level; 
they established the existence and the asymptotic behavior of non-negative solutions to the problem. Furthermore, the existence of a solution for different critical fractional Kirchhoff problems set
on the whole space $\mathbb{R}^N$ is given by Liang \& Shi~\cite{liang2013soliton}.
More recently, Chen~\cite{chen2019existence} establish the existence of solutions to the fractional $p$-Kirchhoff type equations with a generalized Choquard nonlinearities without assuming the Ambrosetti Rabinowitz condition.

\subsection{The kinds and varieties of potential functions}
Several hypotheses have been used on the potential function included in the class of elliptical problems. For example, 
Berestycki \& Lions~\cite{berestycki1983nonlinear} considered the case in which the potential function $V\colon \mathbb{R}^N \to \mathbb{R}$ is constant. Afterwards, Pankov~\cite{pankov2004periodic} studied a problem with $V \in L^{\infty}$ being a periodic function with unit period in each variable, that is, $V(x+z)=V(x)$ for all $x\in \mathbb{R}^N$ and with $z \in \mathbb{Z}^N$; for other articles about problems with periodic potentials, see Coti-Zelati \& Rabinowitz~\cite{sissa1992homoclinic} and Kryszewski \& Szulkin~\cite{kryszewski1998generalized}. Zhu \& Yang~\cite{jianfu1987existence} studied problems with an asymptotic potential V at a positive constant, i.e., there is $V_{\infty} \in \mathbb{R}_+$ such that $|V(x)-V_{\infty}| \to 0$ when $|x|\to +\infty$ and $V(x) \leqslant V_{\infty}$ for all $x \in \mathbb{R}^N$. For the case where the potential $V$ is strictly positive and the Lebesgue measure of the set $\{x \in \mathbb{R}^N \colon V(x)\leqslant M\}$ is finite for all $M \in \mathbb{R}_+$, see Bartsch \& Wang~\cite{bartsch1995existence}. Costa~\cite{costa1994class} and Miyagaki~\cite{miyagaki1997class} studied the case of a coercive potential  $V \colon \mathbb{R}^N \to \mathbb{R}$, that is, $\lim_{|x| \to + \infty} V(x)=+\infty$. For the case of a radial potential $V\colon \mathbb{R}^N \to \mathbb{R}$, i.e., $V(x)=W(r)$ such that $W \colon \mathbb{R}^*_+ \to \mathbb{R}$ and $r=|x|$ for all $x \in \mathbb{R}^N$, see Alves, de Morais Filho \& Souto~\cite{alves1996radially}. Alves \& Souto~\cite{alves2012existence,alves2013existence} considered a continuous, non-negative potential function $V$  which can vanish at infinity, that is, $V(x)\to 0$ when $|x| \to \infty$; for other problems involving this kind of potential, see Alves, Assunção \& Miyagaki~\cite{alves2015existence} and Alves \& Assunção~\cite{alves2022existence}.

\subsection{Palais-Smale and Cerami conditions}

Let $B$ be a Banach space such that $J \colon B \to \mathbb{R}$ is a $C^1$ functional defined on $B$ and $(u_n)$ is a sequence in $B$. The Palais-Smale condition $(PS)_c$ at level $c$ means that if the sequence $\{u_n\}_{n\in\mathbb{N}}$ is such that $J(u_n)\to c$ and $J'(u_n) \to 0$ as $n\to+\infty$, then $\{u_n\}$ has a convergent subsequence. The Cerami condition $(C)_c$ at level $c$ means that if the sequence$\{u_n\}_{n\in\mathbb{N}}$ is such that $J(u_n)\to c$ and $(1+\|u_n\|)J'(u_n) \to 0$ as $n\to+\infty$, then $c\in \mathbb{R}$ is a critical value of $J$. With the above conditions it is possible to show that if a sequence $\{u_n\}_{n\in\mathbb{N}}$ verifies the Palais-Smale condition $(PS)_c$, then it also verifies the Cerami condition $(C)_{c}$; for more details see Costa~\cite{costa2023remarks}. That the Cerami condition $(C)_{c}$ does not imply the Palais-Smale condition $(PS)_c$ can be seen by the function $z \colon \mathbb{R}^{2}\to\mathbb{R}$ defined by $z(x,y)=\ln\big(1+x^{2}\big)-\ln\big(1+y^{2}\big)$; this function verifies the Cerami condition $(C)_{0}$ but does not verify the Palais-Smale condition $(PS)_{0}$, for the level set $z^{-1}(0)$ is $|x|=|y|$; see Robinson~\cite{robinson1995multiple}. A Cerami sequence can produce a critical point even when a $(PS)$ sequence does not. 
A condition similar to $(C)_c$ was introduced by Cerami and was applied to the search for critic points of a functional on an unbounded Riemannian manifold. It should be mentioned that this weakening of the Palais-Smale condition seems essential in the study of variational problems in the strong resonance case because in general the Palais-Smale condition is not satisfied. For more information on these kinds of compactness conditions see the following comments about the nonlinearities.

\subsection{Some types of frequently used nonlinearities}

Several interesting questions arise when we consider the nonlinearities that appear in the study of partial differential equations. For example, inspired by Harrabi~\cite{harrabi2014palais}, consider the general prototype equation $-\Delta u=f(x,u)$ where $x\in\Omega$ with $\Omega\subset\mathbb{R}^{N}$ a bounded, open subset. We look for weak solutions in the Sobolev space $H_{0}^{1}(\Omega)\coloneqq \left\{w\in H^{1}(\Omega)\colon w=0 \text{ on } \Omega\right\}$. As usual, a weak solution to this problem is any function $u\in H_{0}^{1}(\Omega)$ such that $\langle u, v \rangle_{H_{0}^{1}(\Omega)} = \int_{\Omega} f(x,u)v\dd{x}$ for every function $v\in H_{0}^{1}(\Omega)$; here, the inner product is defined by $\langle u, v \rangle_{H_{0}^{1}(\Omega)}\coloneqq \int_{\Omega} uv \dd{x}$. It is well known that a function $u\in H_{0}^{1}(\Omega)$ is a weak solution to this problem if, and only if, it is a critical point of the Euler-Lagrange energy functional defined by $J(u)=(1/2)\|u\|_{H_{0}^{1}(\Omega)}^{2}-\int_{\Omega} F(x,u)\dd{x}$ where $F(x,s)=\int_{0}^{s}f(x,t)\dd{t}$. 

One can ask whether the differential equations have any nontrivial solutions; one can also ask whether it is possible to give a lower bound to the number of solutions by using some topological facts related to the nonlinearity $f\colon \Omega\times \mathbb{R}\to\mathbb{R}$ (eg. \xspace if it is odd in the first variable). For example, see Amann \& Zehnder~\cite{amann1980nontrivial} or Castro \& Lazer~\cite{castro1979critical}.

For general operators, the blow up argument can be used to get existence of positive solution when the nonlinearity $f$ has an assymptotical behavior like $f(s)=|s|^{q-2}s$ at infinity with $1<q<2^{*}=2N/(N-2)$. 

Most results use some hypotheses on the nonlinearity $f$ to make variational methods work. For example, $f \in C(\Omega \times \mathbb{R};\mathbb{R})$ satisfying the large subcritical growth condition,
\begin{enumerate}[label=($h\sb{\arabic*}$), 
ref=$h\sb{\arabic*}$]
\item \label{item:harrabi:h} there exist $C_{0}\in\mathbb{R}_{+}$ and $s_{0}\in\mathbb{R}_{+}$ such that $|f(x,s)|\leqslant C|s|^{2^{*}-1}$ for every $|s| \geqslant s_{0}$ and for every $x\in\Omega$.
\end{enumerate}

Under hypothesis~\eqref{item:harrabi:h} the energy  functional is well defined in the Sobolev space $H_{0}^{1}(\Omega)$ and belongs to $C^{1}(H_{0}^{1}(\Omega);\mathbb{R})$. If we impose some more additional conditions on the nonlinearity $f$, for example if $f(x,s)=a(x)g(s)$ where $a\in C^{0,\alpha}(\overline{\Omega})$ and 
$g\in C^{0,\alpha}_{\operatorname{loc}}(\mathbb{R})$, then it is possible to prove that any weak solution to the problem belongs to the space $C^{2}(\overline{\Omega})$.
 
The use of critical point theory needs a compactness condition, usually the Palais-Smale $(PS)$ condition or the Cerami $(C)$ condition. Our main goal in this brief survey is to revise the required standard assumptions used to get one these conditions. See Cl\'{e}ment, Figueiredo \& Mitidieri~\cite{clement} or Ramos \& Rodrigues~\cite{ramos2001fourth}.

In great part of the literature the $(PS)$ condition is proved by using standard assumptions. Mainly, the Ambrosetti-Rabinowitz condition, the $(AR)$ condition in short, which supposes the existence of $\xi\in\mathbb{R}$ with $\xi > 2$ and $s_{0}\in\mathbb{R}_{+}$ such that $sf(x,s)\geqslant \xi F(x,s)$ for $|s|>s_{0}$ and $x\in \Omega$.

Another typical hypothesis is the subcritical polynomial growth condition,
\begin{enumerate}[label=($h\sb{\arabic*}$), 
ref=$h\sb{\arabic*}$,resume]
\item \label{item:harrabi:scp} there exist $C\in\mathbb{R}_{+}$ and $q\in\mathbb{R}$ such that $1\leqslant q < 2^{*}$ such that $|f(x,s)|\leqslant C (|s|^{q-1}+1)$ for every $x\in\Omega$ and for every $s\in\mathbb{R}$. 
\end{enumerate}
Under the $(AR)$ condition and hypothesis~\eqref{item:harrabi:scp}, the Euler-Lagrange energy functional associated to the differential equation verifies the $(PS)$ condition. 

\subsection{The Ambrosetti-Rabinowitz condition revisited}
The major difficulty in the use of the $(PS)$ condition consists often in proving the boundedness of any Palais-Smale sequence $\{ u_{n} \}_{n\in\mathbb{N}}\subset H_{0}^{1}(\Omega)$. In contrast, for the $(AR)$ condition one has $(\xi/2-1)\| u \|^{2} \leqslant C_{0}(\| u \| + 1)$. However, the $(AR)$ condition is too restrictive and one requires instead the strong superlinear condition,
\begin{enumerate}[label=($h\sb{\arabic*}$), 
ref=$h\sb{\arabic*}$,resume]
\item \label{item:harrabi:3} there exists $C\in\mathbb{R}_{+}$ and $q\in\mathbb{R}$ with $q=\xi-1>0$ such that $|f(x,s)| \geqslant C(|s|^{q}-1)$ for every $x\in\Omega$ and for every $s\in\mathbb{R}$.
\end{enumerate}
In the particular case of the Sobolev space $H_{0}^{1}(\Omega)$, many new existence results have been obtained when $(AR)$ is relaxed by condition~\eqref{item:harrabi:3}. Therefore, some mild oscillations of the nonlinearity $f$ can be allowed. See de Figueiredo \& Yang~\cite{de2003semilinear} and Jeanjean~\cite{jeanjean1999existence}.

However, condition~\eqref{item:harrabi:3} is also violated by many nonlinearities, as for example, $f(s)=as$ or $f(s)=as\ln(s)$ at infinity with $a\in\mathbb{R}_{+}$.
Some special attention has been given to the value $\xi=2$ to introduce weaker condition than $(AR)$ and no longer require the strong superlinear condition. For example, 
\begin{enumerate}[label=($h\sb{\arabic*}$), 
ref=$h\sb{\arabic*}$,resume]
\item \label{item:harrabi:h1} there exist $c\in\mathbb{R}_{+}$ and $s_{0}\in\mathbb{R}_{+}$ such that $c|f(x,s)|^{2N/(N+2m)} \leqslant sf(x,s)-2F(x,s)$ for every $|s| > s_{0}$ and for every $x\in\Omega$.
\end{enumerate}
The key ingredient in this approach is the Riesz-Fréchet representation theorem, which permits one to write $J'(u_{n})$ as a variational equation by supposing the existence of $v_{n} \in H_{0}^{1}(\Omega)$ such that 
$J'(u_{n})\phi = \langle v_{n}, \phi \rangle_{H_{0}^{1}(\Omega)}$ for every $\phi \in H_{0}^{1}(\Omega)$ and $|J'(u_{n})|_{(H_{0}^{1}(\Omega))'}=|v_{n}|_{H_{0}^{1}(\Omega)}$. Thus, $u_{n}-v_{n}$ could be seen as a weak solution in $H_{0}^{1}(\Omega)$ of the equation $\langle u_{n}-v_{n}, \phi \rangle_{H_{0}^{1}(\Omega)} = \int_{\Omega} f(x,u_{n}(x)) \phi \dd{x}$ for every $\phi \in H_{0}^{1}(\Omega)$. Another new aspect in this argument is the use of the Lesbesgue space theory to show the boundedness of the sequence $\{u_n\}_{{n\in \mathbb{N}}} \subset H_{0}^{1}(\Omega)$. To accomplish this, a regularity result due to Agmon, Douglis \& Nirenberg~\cite{agmon1959estimates} can be useful.

Notice that from condition~\eqref{item:harrabi:h}, $C_{0}|f(x,s)|^{2^{*}}\leqslant |sf(x,s)|$; and from the $(AR)$ condition, $0<(1-2/\xi)sf(x,s) \leqslant s f(x,s) - 2F(x,s)$ for every $|s|>s_{0}$ and for every $x\in\Omega$. Hence, condition~\eqref{item:harrabi:h1} is weaker than the $(AR)$ condition. On the other hand, if $F(x,\pm s_{0})>0$, then condition~\eqref{item:harrabi:h1} implies condition~\eqref{item:harrabi:h}.
 
From condition~\eqref{item:harrabi:h1} it follows that 
$\int_{\Omega} |f(x,u_{n})|^{2^{*}}=O(\|u\|+1)$ for Palais-Smale sequences. So, it seems that condition~\eqref{item:harrabi:h1} is an optimal condition ensuring the boundedness of the sequence.

The function $f_{\alpha}(s)=s [g(|s|)]^{\alpha}$ where $g(s)=\ln(\ln|s|)$ verifies condition~\eqref{item:harrabi:h1} for every $\alpha \in\mathbb{R}_{+}$; however, it does not verify the strong superlinear condition~\eqref{item:harrabi:3}. Moreover, $f(s)=as$ does not verify condition~\eqref{item:harrabi:h1}; however, $f(s)=as-|s|^{\alpha - 1}s$ with $(2^{*}-1)^{-1}\leqslant \alpha < 1$ and $f(s)=as+s[\ln(|s|+2)]^{-\alpha'}$ with $\alpha'\in\mathbb{R}_{+}$ verify condition~\eqref{item:harrabi:h1} but not condition~\eqref{item:harrabi:3}.

\subsection{Subcritical polynomial growth condition revisited}
Upon verifying the boundedness of the $(PS)$ sequences, 
the use of the compactness of the embedding $H_{0}^{1}(\Omega) \hookrightarrow L^{q}(\Omega)$ together with condition~\eqref{item:harrabi:scp} allows one to prove that if the sequence $\{u_{n}\}_{n\in\mathbb{N}}\subset H_{0}^{1}(\Omega)$ is bounded, then $f(x,u_{n})$ has a convergent subsequence in $L^{2N/(N+2)}(\Omega)$. This means that the operator $K(u)v=\int_{\Omega} f(x,u)v\dd{x}$ is compact. But this condition is not satisfied when the nonlinearity is very close to critical growth, as in the example $f_{\alpha}(s)=|s|^{4/(N-2)}s/\ln^{\alpha}(|s|+2)$ for $\alpha\in\mathbb{R}_{+}$. However, the operator $K$ is compact for $f_{\alpha}$, which means that condition~\eqref{item:harrabi:scp} is only a sufficient condition. One way to weaken this condition is to substitute it with the condition
\begin{enumerate}[label=($h\sb{\arabic*}$), 
ref=$h\sb{\arabic*}$,resume]
\item \label{item:harrabi:h2} $\lim_{s\to+\infty} f(x,s)/|s|^{2^{*}-1}=0$ uniformly with respect to $x\in\Omega$.
\end{enumerate}
The operator $K$ is still compact under this assumption. Moreover, the condition~\eqref{item:harrabi:h2} seems to be nearly optimal because if $f(s)=|s|^{4/(N-2)}s$ at infinity, then $K$ is no longer compact.

Since many existence results are based on the fact that the $(PS)$ condition is satisfied, most cases require the $(AR)$ condition as well as the subcritical polynomial growth. Thus, after verifying $(PS)$ condition under hypotheses~\eqref{item:harrabi:h1} and~\eqref{item:harrabi:h2}, it is possible to improve some classical existence results having the minimax structure. For example, let $\lambda_{1}$ be the lowest eigenvalue of the self-adjoint $(-\Delta)u=f(x,u)$ problem with Dirichlet condition. Then the energy functional has a nontrivial critical point by the mountain pass theorem if $\lim_{s\to+\infty}f(x,s)>\lambda_{1}$ uniformly in $\overline{\Omega}$ and $\lim_{s\to 0} f(x,s)/s < \lambda_{1}$ uniformly in $\overline{\Omega}$.

Consider the function 
$f(s)=
 |s|^{4/(N-2)}s/\ln^{q+q'}(|s|)
+[|s|^{4/(N-2)}s/\ln^{q}(|s|)][\gamma+\sin(\ln(|s|))]$, defined for $|s|>1$, where $q\in\mathbb{R}_{+}$ and $0 < q' < 1$. It can be shown that there exists a constant $\underline{\gamma}\in\mathbb{R}$ such that for $\gamma > \underline{\gamma}$, then $f$ verifies the $(AR)$ condition; and if $\gamma < \underline{\gamma}$, then $f$ does not verify neither the $(AR)$ nor the \eqref{item:harrabi:h1} conditions. However, if $\gamma = \underline{\gamma}$ and if $q'< \min \{ 1, q (N-2)/(N+2) \}$, then $f$ does not verify neither the $(AR)$ nor the \eqref{item:harrabi:scp} conditions but verifies both conditions~\eqref{item:harrabi:h1} and~\eqref{item:harrabi:h2}. This shows some improvements brought by these conditions.

\subsection{Subcritical polynomial growth and Cerami conditions}
Consider again our prototype equation $-\Delta u = f(x,u)$ in the bounded, open domain $\Omega\subset\mathbb{R}^{N}$ with a variational structure; the energy functional $J\colon H_{0}^{1}(\Omega)\to\mathbb{R}$ associated to this problem can be defined by $J(u)=(1/2)\|u\|^{2}-\int f(x,u(x))\dd{x}$. 

If the Cerami condition is verified and since $\|J'(u_{n})\|_{(H_{0}^{1}(\Omega))'} \| u_{n}\| \to 0$ as $n\to+\infty$, then every Cerami sequence satisfies $2J(u_{n})-J'(u_{n})u_{n}=O(1)$, contrarily to the Palais-Smale sequences, where one only has $2J(u_{n})-J'(u_{n})u_{n}=O(\|u_{n}\|+1)$.
For instance, in case $f(s)=a(x)s+b|s|^{\alpha-1}s$ or $f(s)=a(x)s+b\log(|s|+1)$ where $a\in C(\overline{\Omega; \mathbb{R}})$ is a continuous, positive function with $b$, $\alpha\in\mathbb{R}$ and $0<\alpha<1$, the energy functional verifies the Cerami condition.

The use of the Cerami condition with the sequence $\{u_{n}\}_{n\in\mathbb{N}}\subset H_{0}^{1}(\Omega)$ usually goes as follows. A common assumption used by some authors is that $H(x,s)=2F(x,s)-sf(x,s) \geqslant -w_{1}(x)$ for $x\in\Omega$ and $t\in\mathbb{R}$, where $w_{1}\in L^{1}(\Omega)$ and that $H(x,s) \to +\infty$ a.e.\xspace as $|s|\to+\infty$; it is possible to prove that $\|u_{n}\|^{2}-\langle f(\cdot,u_{n}),u_{n}\rangle \to 0$ as $n\to +\infty$ and this implies that $\int_{\Omega}H(x,u_{n})\dd{x}\leqslant K$ for some constant $K\in\mathbb{R}$ . Then, towards a contradiction, it is assumed that the sequence $\{\|u_{n}\|\}_{n\in\mathbb{N}}\subset \mathbb{R}$ is unbounded, i.e., \xspace $\| u_{n} \| \to +\infty$ as $n \to +\infty$. Another sequence is now created by defining $\Tilde{u}_{n}(x)\coloneqq u_{n}(x)/\|u_{n}\|$; therefore, $\| \Tilde{u}_n \| = 1$ and, up to passage to a subsequence, we have $\Tilde{u}_n \rightharpoonup \Tilde{u}$ weakly in $H_{0}^{1}(\Omega)$, $\Tilde{u}_n \to \Tilde{u}$ strongly in $L^{2}(\Omega)$, and $\Tilde{u}_n \to \Tilde{u}$ a.e.\xspace in $\Omega$; moreover, it can be showed that $\Tilde{u}\not\equiv 0$. If we denote $\Omega_{0}\coloneqq \{ x\in\Omega\colon \Tilde{u}(x)\neq 0 \}$ and $\Omega_{1}\coloneqq \Omega \backslash \Omega_{0}$, then $|u_{n}(x)|=\|u_{n}\| \Tilde{u_{n}}(x) \to +\infty$ as $n\to +\infty$ for every $x\in \Omega_{0}$ and 
\begin{align*}
    \int_{\Omega_{0}\cup \Omega_{1}} H(x,u_{n}(x))\dd{x} \geqslant \int_{\Omega_{0}} H(x,u_{n}(x))\dd{x}-\int_{\Omega_{1}}w_{1}(x)\dd{x} \to +\infty.
\end{align*}
But this contradicts the boundedness of the lefthand side integral previously mentioned. Therefore, the Cerami sequence $\{u_{n}\}_{n\in\mathbb{N}}\subset H_{0}^{1}(\Omega)$ must be bounded and we obtain some compactness to work with in the proof of the existence result.

The crucial element in this argument is the estimate $\|u_{n}\|^{2}-\langle f(\cdot,u_{n}),u_{n}\rangle \to 0$ as $n\to +\infty$. If we had been dealing with a Palais-Smale sequence all the time, we could only conclude that $\|u_{n}\|^{2}-\langle f(\cdot,u_{n}),u_{n}\rangle = o(\| u_{n} \|)$ which would only imply that $\int_{\Omega} H(x,u_{n})\dd{x}=o(\| u_{n} \|)$. This would not contradict the estimate $\int_{\Omega} H(x,u_{n}(x))\dd{x} \to +\infty$ as $n\to+\infty$ and the argument would not go through. For more details, see Schechter~\cite{Schechter}.

\subsection{Degraded oscillation case}
Consider the nonlinearity with a very sparsed oscilation, for example, $f(s)=\gamma s^{p}+s^{p}(1+\sin(\log(s+2)))$ if $s\geqslant 0 $ and $f(s)=0$ if $s<0$. Then $f$ satisfies the subcritical polynomial growth condition at infinity for every $\gamma\geqslant 0$. Moreover, there exists $\underline{\gamma} \in\mathbb{R}_{+}$ such that if $\gamma > \underline{\gamma}$ then $f$ verifies the $(AR)$ condition and if $0<\gamma\leqslant \underline{\gamma}$ then $f$ verifies only the strong superlinear condition. However, if $\gamma=0$ then $f$ does not even verify the condition $\lim_{s\to+\infty} f(s)/s=+\infty$ uniformly with respect to $x \in \Omega$ since $f(\exp(\exp((2n-1/2)\pi-2)))=0$. This case is referred to as the degraded oscillation case. Under the strong superlinear condition together using the assumption $\lim_{s\to+\infty} [sf'(s)-pf(s)]/s^{p}=0$ and some additional conditions it is possible to prove the existence of at least one positive solution. For example, $f(s)=s^{p}(1+\sin(\log(\log(s+2))))$ is an instance for this situation. However the last condition is too restrictive; for example, the function $f(s)=\gamma s^{p}+s^{p}\sin(\log(s+2))$ for $s\geqslant0$ with $\gamma >1$ does not verify it.

\subsection{The case of resonant nonlinearities.}
Other kind of question is related to resonant nonlinearities. More precisely, consider an asymptotically linear function, that is, $\lim_{|t|\to+\infty} t^{-1}f(t)=\alpha$ where $\alpha\in\mathbb{R}$ is finite; then we can write $f(t)=\alpha t - g(t)$ with $\lim_{|t|\to_\infty}t^{-1}g(t)=0$ where $g\colon\mathbb{R}\to\mathbb{R}$. As usual, we denote by $\lambda_{1} < \lambda_{2} < \cdots < \lambda_{k} < \cdots$ the sequence of eigenvalues of the self-adjoint realization in $L^{2}(\Omega)$ of the Laplacian operator with Dirichlet boundary condition. We say that the problem is resonant at infinity if $\alpha = \lambda_{k}$ for some positive integer $k\in\mathbb{N}$. Depending on the growth of the function $g$ at infinity we have different degrees of resonance; that is, the smaller the $g$, the stronger the resonance. To quantify these degrees of resonance, we can consider some situations: 
\begin{enumerate}[label=($l\sb{\arabic*}$), 
ref=$l\sb{\arabic*}$]
\item $\lim_{t\to+\infty}g(t)=\ell_{+}\neq 0$ and $\lim_{t\to-\infty}g(t)=\ell_{-}\neq 0$: this weak resonance was first considered by Landesman \& Lazer~\cite{landesman1970nonlinear};

\item $\lim_{|t|\to+\infty}g(t)=0$ and 
$\lim_{t\to+\infty} \int_{0}^{t}g(s)\dd{s}=\pm\infty$: this mild resonance was first considered by Ahmad, Lazer \& Paul~\cite{ahmad1976elementary}; 

\item $\lim_{|t|\to+\infty}g(t)=0$ and 
$\lim_{|t|\to+\infty} \int_{0}^{t}g(s)\dd{s}=\beta$ where $\beta\in\mathbb{R}$ is finite: this strong resonance at infinity was considered by Thews~\cite{thews1980non};

\item $\lim_{|t|\to+\infty}tg(t)=0$, 
$\lim_{t\to+\infty} \int_{-\infty}^{t}g(s)\dd{s}=0$, this integral being well-defined and non-negative for every $t\in\mathbb{R}$: this strong resonance at infinity was considered by Bartolo, Benci \& Fortunato~\cite{bartolo1983abstract}. 
\end{enumerate}

In general terms, the existence results mentioned are proved through the application of deformation lemmas whose proofs, in turn, rely on a weakened version of the well-known Palais-Smale condition introduced by Cerami.

To conclude, we mention that when the variational approach can not be employed, the question of existence of solutions may be dealt with via topological methods. In this case, the proof of existence of solutions is essentially reduced to deriving {\itshape a priori} estimates for all possible solutions and in general needs that the domain $\Omega\subset\mathbb{R}^{N}$ be convex or a ball. However, certain behavior of the nonlinearity at infinity is still necessary.

For these and several other existence results above mentioned, see the interesting paper by Harrabi~\cite{harrabi2014palais}.

\subsection{Our contribution to the problem}
Motivated by the above papers, our results improve upon previous work in the following ways: we focus our attention on the existence of a nontrivial weak solution for fractional $p$-Kirchhoff equation in the entire space $\mathbb{R}^{N}$, which causes a difficulty due to
lack of compactness for Sobolev theorem;
 moreover, the problem also includes a nonlocal Choquard subcritical term and a critical Hardy-type term; additionally, we consider  singularities in the fractional $p$-Laplacian with 
$\theta = \theta_1 + \theta_2 $ not necessarily zero and we also add a critical Sobolev nonlinearity. The possibility of a slower growth in the nonlinearity makes it more difficult to establish a compactness condition. In fact, we will not prove the usual Palais-Smale condition, but rather a less restrictive version often credited to Cerami. Our argument has two crucial points: the first one is to prove a uniform boundedness of the convolution part $|I_{\mu}\ast F|<+\infty$, which gives a lot of help when we choose Cerami sequences; the second one to treat the lack of compactness of the Sobolev embeddings.

\section{The variational setting}
Here we recall a generalization of the Hardy-Littlewood-Sobolev, also called the doubly weighted inequality or the Stein-Weiss inequality. It provides quantitative information to characterize the integrability for the integral operators present in the energy functional and is crucial in the analysis developed in this work. See
Stein \& Weiss~\cite{MR0098285}; Lieb~\cite{MR717827} and Lieb \& Loss~\cite[Theorem~4.3]{MR1817225}.

\begin{proposition}[Doubly weighted Stein-Weiss inequality]
\label{prop2.1:chen}
Let $1 < r, \; t < +\infty$,
$0 < \mu < N$, 
and 
$\eta + \kappa \geqslant 0$
such that
$\mu + \eta + \kappa \leqslant N$,
$\eta < N/r'$,
$\kappa < N/t'$
and 
$1/t+(\mu + \eta + \kappa)/N+1/r=2$;
let $f \in L^t(\mathbb{R}^N)$ and 
$h \in L^r(\mathbb{R}^N)$. Then there exists a constant $C(N, \mu, r, t, \eta, \kappa)$, independent on $f$ and $h$, such that
\begin{align}
    \label{des:steinweiss}
\left|      \iint_{\mathbb{R}^{2N}} 
\frac{{f(x)}h(y)}{|x|^{\eta}|x-y|^\mu|y|^{\kappa}}\dd x \dd y\right| 
& \leqslant C(N, \mu, r, t, \eta, \kappa) \|f\|_{L^t(\mathbb{R}^N)} \|h\|_{L^r(\mathbb{R}^N)}.
\end{align}
\end{proposition}

\begin{corollary}
\label{teo2.1:chen}
Let $0 < s < 1$;
$ 0 \leqslant \alpha < sp + \theta < N$;
$0 < \mu < N$;
given a function
$ u \in \Dot{W}^{s,p}_{\theta}(\mathbb{R}^N)$ 
consider Proposition~\ref{prop2.1:chen} with 
$\eta=\kappa=\delta$; 
$2\delta + \mu \leqslant N$ and 
$t=r=N/(N - \delta -\mu/2)$.
Then  
$f, h \in L^{\frac{N}{N-\delta-\mu/2}}(\mathbb{R}^N)$ and
\begin{align}
    \label{des:2.6chineses:bisss}
        \iint_{\mathbb{R}^{2N}} \frac{|f(x)||h(y)|}{|x|^{\delta}|x-y|^{\mu}|y|^{\delta}}\dd x\dd y
    & \leqslant C(N,\delta,\theta,\mu)
    \|f\|_{L^t(\mathbb{R}^N)} \|h\|_{L^t(\mathbb{R}^N)}.
\end{align}
\end{corollary}

In general, for $\eta = \kappa = \delta$ and $t = r$, the map
\begin{align*}
    u & \mapsto 
    \iint_{\mathbb{R}^{2N}}
    \dfrac{|u(x)|^{q}|u(y)|^{q}}{|x|^{\delta}|x-y|^{\mu}|y|^{\delta}} \dd{x} \dd{y}
\end{align*}
is well defined by inequalities~\eqref{def:psus}.

The variational structure of 
problem~\eqref{problema:Kirchhof} can be established with the help of several results. To ensure the well-definiteness of the energy functional, we use the following result.

\begin{lemma}\label{lema2.2:chen}
Let \eqref{hip:V} and \eqref{hip:m1} hold. Then the functional $\Phi$ defined in~\eqref{kirchhoff:funcionalI} is of class $C^1(W^{s,p}_{V,\theta}(\mathbb{R}^N),\mathbb{R})$ and
\begin{align*}
\langle \Phi'(u),\varphi\rangle&=m(\|u\|_W^p)\Big[\iint_{\mathbb{R}^{2N}}\frac{|u(x)-u(y)|^{p-2}(u(x)-u(y))(\varphi(x)-\varphi(y))}{|x|^{\theta_1}|x-y|^{N+ps}|y|^{\theta_2}}\dd xdy\\
&\qquad\qquad+\int_{\mathbb{R}^N}V(x)\frac{|u(x)|^{p-2}u(x)\varphi(x)}{|x|^{\alpha}}\dd x\Big],
\end{align*}
for all $u,\varphi\in W^{s,p}_{V, \theta}(\mathbb{R}^N)$. Moreover, $\Phi$ is weakly lower semi-continuous in $W^{s,p}_{V, \theta}(\mathbb{R}^N)$.
\end{lemma}

\begin{proof}
    Let $\{u_n\}_n \subset W$ and $u \in  W$ satisfy $u_n \to u$ strongly in $W$ as $n \to \infty$.
Without loss of generality, we assume that $u_n \to u$ a.e. in $\mathbb{R}^N$. Then the sequence
\begin{align}
    \Bigl\{\dfrac{|u_n(x)-u_n(y)|^{p-2}(u_n(x)-u_n(y))}{|x|^{\theta_1/p'}|x-y|^{(N+sp)/p'}|y|^{\theta_2/p'}}\Bigr\}_n \quad \text{is bounded in $L^{p'}(\mathbb{R}^{2N})$},
\end{align}
as well as in $\mathbb{R}^{2N}$
\begin{align*}
   U_n(x,y) &\coloneqq \dfrac{|u_n(x)-u_n(y)|^{p-2}(u_n(x)-u_n(y))}{|x|^{\theta_1/p'}|x-y|^{(N+sp)/p'}|y|^{\theta_2/p'}} \\
   & \to  \dfrac{|u(x)-u(y)|^{p-2}(u(x)-u(y))}{|x|^{\theta_1/p'}|x-y|^{(N+sp)/p'}|y|^{\theta_2/p'}} \coloneqq U(x,y).
\end{align*}
Thus, the Brezis–Lieb lemma implies
\begin{align}
\label{eq0}
     &\lim\limits_{n \to \infty}  \iint _{\mathbb{R}^{2N}} |U_n(x,y)-U(x,y)|^{p'} \dd x \dd y \nonumber\\ 
      & \quad = \lim\limits_{n \to \infty}  \iint _{\mathbb{R}^{2N}} (|U_n(x,y)|^{p'} - |U(x,y)|^{p'}) \dd x \dd y \nonumber\\
     & \quad =  \lim\limits_{n \to \infty}  \iint _{\mathbb{R}^{2N}} \Bigl(\frac{|u_n(x)-u_n(y)|^p}{|x|^{\theta_1}|x-y|^{N+sp}|y|^{\theta_2}} - \frac{|u(x)-u(y)|^p}{|x|^{\theta_1}|x-y|^{N+sp}|y|^{\theta_2}} \Bigr) \dd x\dd y.
\end{align}
The fact that $u_n \to u$ strongly in W yields that
\begin{align}
\label{eq1}
    \lim\limits_{n \to \infty}  \iint _{\mathbb{R}^{2N}} \Bigl(\frac{|u_n(x)-u_n(y)|^p}{|x|^{\theta_1}|x-y|^{N+sp}|y|^{\theta_2}} - \frac{|u(x)-u(y)|^p}{|x|^{\theta_1}|x-y|^{N+sp}|y|^{\theta_2}} \Bigr) \dd x\dd y = 0.
\end{align}
Moreover, the continuity of $m$ implies that
\begin{align}
    \label{eq2}
    \lim\limits_{n \to \infty} m ([u_n]^p_{W^{s,p}_{ \theta}(\mathbb{R}^N)}) = m ([u]^p_{W^{s,p}_{\theta}(\mathbb{R}^N)}).
\end{align}
From~\eqref{eq0} it follows that
\begin{align}
\label{eq22}
    \lim\limits_{n \to \infty}  \iint _{\mathbb{R}^{2N}} |U_n(x,y)-U(x,y)|^{p'} \dd x \dd y = 0.
\end{align}
Similarly, the sequence 
\begin{align}
    \Bigl\{\dfrac{V^{1/p'}(x)|u_n(x)|^{p-2}u_n(x)}{|x|^{\alpha/p'}}\Bigr\}_n \quad \text{is bounded in $L^{p'}(\mathbb{R}^{2N})$},
\end{align}
as well as in $\mathbb{R}^{2N}$
\begin{align*}
   K_n(x,y) &\coloneqq \dfrac{V^{1/p'}(x)|u_n(x)|^{p-2}u_n(x)}{|x|^{\alpha/p'}} 
   \to  \dfrac{V^{1/p'}(x)|u(x)|^{p-2}u(x)}{|x|^{\alpha/p'}} \coloneqq K(x,y).
\end{align*}
Thus, the Brezis–Lieb lemma implies
\begin{align}
\label{eq01}
     \lim\limits_{n \to \infty} \int _{\mathbb{R}^{N}} |K_n(x,y)-K(x,y)|^{p'} \dd x & = \lim\limits_{n \to \infty}  \int _{\mathbb{R}^{N}} (|K_n(x,y)|^{p'} - |K(x,y)|^{p'}) \dd x \nonumber \\ 
     & =  \lim\limits_{n \to \infty}  \int _{\mathbb{R}^{N}} \Bigl(\frac{V(x)|u_n(x)|^p}{|x|^{\alpha}} - \frac{V(x)|u(x)|^p}{|x|^{\alpha}} \Bigr) \dd x.
\end{align}
The fact that $u_n \to u$ strongly in W yields that
\begin{align}
\label{eq11}
    \int _{\mathbb{R}^{N}} \Bigl(\frac{V(x)|u_n(x)|^p}{|x|^{\alpha}} - \frac{V(x)|u(x)|^p}{|x|^{\alpha}} \Bigr) \dd x = 0.
\end{align}
From~\eqref{eq01} it follows that
\begin{align}
\label{eq222}
    \lim\limits_{n \to \infty} \int  _{\mathbb{R}^{N}} |K_n(x,y)-K(x,y)|^{p'} \dd x  = 0.
\end{align}
From Hölder inequality, we have
\begin{align*}
& \iint_{\mathbb{R}^{2N}} \frac{|u_n(x)-u_n(y)|^{p-2}(u_n(x)-u_n(y))(\phi (x) - \phi(y))}{|x|^{\theta_1}|x+y|^{N+sp}|y|^{\theta_2}} \dd x \dd y \\
& \quad + \int _{\mathbb{R}^N} \frac{V(x)|u_n(x)|^{p-2}u_n(x)\phi (x)}{|x|^{\alpha}} \dd x \\
   & = \iint_{\mathbb{R}^{2N}} \frac{|u_n(x)-u_n(y)|^{p-2}(u_n(x)-u_n(y))}{|x|^{\frac{\theta_1}{p'}}|x+y|^{\frac{N+sp}{p'}}|y|^{\frac{\theta_2}{p'}}} \cdot \frac{\phi (x)-\phi(y)}{|x|^{\frac{\theta_1}{p}}|x+y|^{\frac{N+sp}{p}}|y|^{\frac{\theta_2}{p}}} \dd x \dd y \\
   & \quad + \int_{\mathbb{R}^N} \frac{(V(x))^{\frac{1}{p'}}|u_n(x)|^{p-2}u_n(x)}{|x|^{\frac{\alpha}{p'}}}\cdot \frac{(V(x))^{\frac{1}{p}}\phi(x)}{|x|^{\frac{\alpha}{p}}} \dd x\\
& \leqslant \Bigl(\iint_{\mathbb{R}^{2N}} \frac{|u_n(x)-u_n(y)|^{p}}{|x|^{\theta_1}|x+y|^{N+sp}|y|^{\theta_2}} \dd x \dd y\Bigr)^{\frac{1}{p'}} \cdot \Bigl(\iint_{\mathbb{R}^{2N}}\frac{|\phi (x)-\phi(y)|^p}{|x|^{\theta_1}|x+y|^{N+sp}|y|^{\theta_2}} \dd x \dd y\Bigr)^{\frac{1}{p}} \\
 & \quad + \Bigl (\frac{V(x)|u_n|^p}{|x|^{\alpha}} \dd x \Bigr)^{\frac{1}{p'}} \cdot \Bigl (\frac{V(x)|\phi |^p}{|x|^{\alpha}} \dd x \Bigr)^{\frac{1}{p}}.
\end{align*}
Similarly, we can obtain the previous result for $u$.

Combining~\eqref{eq2},~\eqref{eq22} and~\eqref{eq222} with the Hölder inequality, we have
\begin{align*}
    \|\Phi '(u_n) - \Phi '(u)\|_{W'} = \sup\limits_{\phi \in W, \; \|\phi \|_W=1}|\langle\Phi'(u_n)-\Phi'(u), \phi\rangle| \to 0
\end{align*}
as $n \to + \infty$. Hence $\Phi \in C^1(W, \mathbb{R})$. Finally, that the map $v 	\mapsto [v]^p_{W^{s,p}_{\theta}(\mathbb{R}^N)}$ is lower semi-continuous in the weak topology of $W^{s,p}_{V, \theta}(\mathbb{R}^N)$ and $M$ is nondecreasing and continuous on $\mathbb{R}_0^+$, so the $u 	\mapsto M([u]^p_{W^{s,p}_{\theta}(\mathbb{R}^N)})$ is lower semi-continuous in the weak topology of $W^{s,p}_{V, \theta}(\mathbb{R}^N)$. 
\end{proof}

One of the main difficulties of this work is to prove the weak strong continuity of the term involving the weighted Sobolev critical exponent. To accomplish this goal, we use the following result.

\begin{lemma}
    \label{sobolev}
    The functional $\Xi$ defined in~\eqref{kirchhoff:funcionalI} as well as $\Xi '$ are weakly strongly continuous on $W_{V, \theta}^{s,p}(\mathbb{R}^N)$. 
\end{lemma}

\begin{proof}
    See 
    %Lemma~\ref{prop:1.3chineses:bis}; see also 
    Assunção, Miyagaki \& Siqueira~\cite[Lemma $1.7$]{assunção2023fractional}.
\end{proof}

\begin{lemma}\label{lema:2.3chen}
Assume (\eqref{hip:F2}) holds, then there exists $K>0$ such that
\begin{align}\label{asd1}
\Bigl|\mathcal{I}_\mu\ast \frac{F(v)}{|x|^{\delta}}\Bigr|\leqslant K\quad \mbox{for}\ v\in W^{s,p}_{V, \theta}(\mathbb{R}^N).
\end{align}
\end{lemma}

\begin{proof}
By the assumptions \eqref{hip:F2} and \eqref{hip:V} and using~\eqref{sobem}, we have
\begin{align*}
\Bigl|\mathcal{I}_\mu\ast \frac{F(v)}{|x|^{\delta}}\Bigr|&=\left|\int_{\mathbb{R}^N}\frac{F(v)}{|x|^{\delta}|x-y|^\mu}\dd y\right|\nonumber\\
\leqslant &\left|\int_{|x-y|\leqslant 1}\frac{F(v)}{|x|^{\delta}|x-y|^\mu}\dd y\right|+\left|\int_{|x-y|\geqslant 1}\frac{F(v)}{|x|^{\delta}|x-y|^\mu}\dd y\right|\nonumber\\
\leqslant&c_0\int_{|x-y|\leqslant 1}\frac{|v|^{q_1}+|v|^{q_2}}{|x|^{\delta}|x-y|^\mu}\dd y+c_0\int_{|x-y|\geqslant 1}\frac{|v|^{q_1}+|v|^{q_2}}{|x|^{\delta}} \dd y\nonumber\\
\leqslant &c_0\int_{|x-y|\leqslant 1}\frac{|v|^{q_1}+|v|^{q_2}}{|x|^{\delta}|x-y|^\mu}\dd y+\frac{c_0}{V_0}\int_{|x-y|\geqslant 1}\frac{V(y)|v|^{q_1}+V(y)|v|^{q_2}}{|x|^{\delta}} \dd y\nonumber\\
\leqslant&c_0\int_{|x-y|\leqslant 1}\frac{|v|^{q_1}+|v|^{q_2}}{|x|^{\delta}|x-y|^\mu}dy+C(\|v\|_W^{q_1}+\|v\|_W^{q_2}) \nonumber\\
\leqslant&c_0\int_{|x-y|\leqslant 1}\frac{|v|^{q_1}+|v|^{q_2}}{|x|^{\delta}|x-y|^\mu}\dd y+C.
\end{align*}

For the first term in the last line above, choosing $t_1\in \Big(\frac{N-ps-\theta}{N-\delta-\mu/2},\frac{p(N-\beta)}{(N-ps- \theta)q_1}\Big)$ and $t_2\in \Big(\frac{N-ps-\theta}{N-\delta - \mu/2},\frac{p(N-\beta)}{(N-ps-\theta)q_2}\Big)$, using H\"{o}lder inequality and the assumption \eqref{hip:V}, we obtain
\begin{align*}
&\lefteqn{\int_{|x-y|\leqslant 1}\frac{|v|^{q_1}+|v|^{q_2}}{|x|^{\delta}|x-y|^\mu}\dd y}\\
& \leqslant \left(\int_{|x-y|\leqslant 1} \frac{|v|^{q_1t_1}}{|x|^{\delta t_1}} \dd y\right)^{\frac{1}{t_1}}\left(\int_{|x-y|\leqslant 1} |x-y|^{-\frac{\mu t_1}{t_1-1}}\dd y\right)^{\frac{t_1-
1}{t_1}}\\
&\quad+\left(\int_{|x-y|\leqslant 1} \frac{|v|^{q_2t_2}}{|x|^{\delta t_2}} \dd y\right)^{\frac{1}{t_2}}\left(\int_{|x-y|\leqslant 1} |x-y|^{-\frac{\mu t_2}{t_2-1}}\dd y\right)^{\frac{t_2-
1}{t_2}}\\
&\leqslant\frac{1}{V_0^{1/t_1}}\left(\int_{|x-y|\leqslant 1} \frac{V(x)|v|^{q_1t_1}}{|x|^{\delta t_1}} \dd y\right)^{\frac{1}{t_1}}\left(\int_{|x-y|\leqslant 1} |x-y|^{-\frac{\mu t_1}{t_1-1}}\dd y\right)^{\frac{t_1-
1}{t_1}}\\
&\quad+\frac{1}{V_0^{1/t_2}}\left(\int_{|x-y|\leqslant 1} \frac{V(x)|v|^{q_2t_2}}{|x|^{\delta t_2}} \dd y\right)^{\frac{1}{t_2}}\left(\int_{|x-y|\leqslant 1} |x-y|^{-\frac{\mu t_2}{t_2-1}}\dd y\right)^{\frac{t_2-
1}{t_2}}
\end{align*}
Let $0< \alpha< N-\mu$, we choose $q_1t_1=p, q_2t_2 =p$ and $\delta t_1=\alpha, \delta t_2=\alpha$, so using~\eqref{sobem}, we have
\begin{align*}
&\lefteqn{\int_{|x-y|\leqslant 1}\frac{|v|^{q_1}+|v|^{q_2}}{|x|^{\delta}|x-y|^\mu}\dd y}\\
    &\leqslant  C\left(\|v\|_W^{q_1} +\|v\|_W^{q_2}\right) \left[\left(\int_{r\leqslant 1} r^{N-1-\frac{\mu t_1}{t_1-1}}\dd y\right)^{\frac{t_1-
1}{t_1}}+\left(\int_{r\leqslant 1} r^{N-1-\frac{\mu t_2}{t_2-1}}\dd y\right)^{\frac{t_2-
1}{t_2}}\right]\\
&\leqslant C.
\end{align*}
\end{proof}

\begin{lemma}\label{lema2a}
Let \eqref{hip:V} and \eqref{hip:F1}--\eqref{hip:F2} hold. Then the functional $\Psi$ defined in~\eqref{kirchhoff:funcionalI} as well as $\Psi'$ are weakly strongly continuous on $W^{s,p}_{V, \theta}(\mathbb{R}^N)$.
\end{lemma}

\begin{proof}
\label{lema:2.4chen}
Let $\{u_n\}$ be a sequence in $W^{s,p}_{V, \theta}(\mathbb{R}^N)$ such that $u_n\rightharpoonup u$ in $W^{s,p}_{V, \theta}(\mathbb{R}^N)$ as $n\to\infty$. Then $\{u_n\}$ is bounded in $W^{s,p}_{V, \theta}(\mathbb{R}^N)$, and then there exists a subsequence denoted by itself, such that
$$
u_n\to u\quad\mbox{in}\ L^{q_1}(\mathbb{R}^N, |x|^{-\delta})\cap L^{q_2}(\mathbb{R}^N, |x|^{-\delta}),\qquad \mbox{and}\ \ \
u_n\to u\quad\mbox{a.e. in}\ \mathbb{R}^N\ \ \mbox{as}\ n\to\infty,
$$
and by \cite[Theorem $4.9$]{brezis2011functional} there exists $\ell\in L^{q_1}(\mathbb{R}^N, |x|^{-\delta})\cap L^{q_2}(\mathbb{R}^N, |x|^{-\delta})$ such that
$$
\frac{|u_n(x)|}{|x|^{\delta}}\leqslant \ell(x)\ \mbox{a.e.\ in}\ \ \mathbb{R}^N.
$$
First, we show that $\Psi$ is weakly strongly continuous on $W^{s,p}_{V, \theta}(\mathbb{R}^N)$. Since $F\in C^1(\mathbb{R},\mathbb{R})$, we see that $\frac{F(u_n)}{|x|^{\delta}}\to \frac{F(u)}{|x|^{\delta}}$ as
$n\to\infty$ for almost all $x\in\mathbb{R}^N$, and so $\Bigl(\mathcal{I}_\mu\ast \frac{F(u_n)}{|x|^{\delta}}\Bigr)\frac{F(u_n)}{|x|^{\delta}}\to \Bigl(\mathcal{I}_\mu\ast \frac{F(u)}{|x|^{\delta}}\Bigr)\frac{F(u)}{|x|^{\delta}}$ as $n\to\infty$  for almost all $x\in\mathbb{R}^N$.
From Lemma \ref{lema:2.3chen} and \eqref{hip:F2}, we have
\begin{align*}
\Bigl|\Bigl(\mathcal{I}_\mu\ast \frac{F(u_n)}{|x|^{\delta}}\Bigr)\frac{F(u_n)}{|x|^{\delta}}\Bigr|
\leqslant Kc_0 \Bigl(\frac{|u_n(x)|^{q_1}}{q_1|x|^{\delta}}+\frac{|u_n(x)|^{q_2}}{q_2|x|^{\delta}}\Bigr)\in L^1(\mathbb{R}^N).
\end{align*}
By Lebesgue dominated convergence theorem, we get
\begin{align*}
\int_{\mathbb{R}^N}\Bigl(\mathcal{I}_\mu\ast \frac{F(u_n)}{|x|^{\delta}}\Bigr)\frac{F(u_n)}{|x|^{\delta}}\dd x\to \int_{\mathbb{R}^N}\Bigl(\mathcal{I}_\mu\ast \frac{F(u)}{|x|^{\delta}}\Bigr)\frac{F(u)}{|x|^{\delta}}\dd x\quad \mbox{as}\ n\to\infty,    
\end{align*}
which implies that $\Psi(u_n)\to\Psi(u)$ as $n\to\infty$. Thus $\Psi$ is weakly strongly continuous on $W^{s,p}_{V, \theta}(\mathbb{R}^N)$.

Next, we prove that $\Psi'$ is weakly strongly continuous on $W^{s,p}_{V, \theta}(\mathbb{R}^N)$. Since $u_n(x)\to u(x)$ as $n\to\infty$ for almost all $x\in\mathbb{R}^N$,
$\frac{f(u_n)}{|x|^{\delta}}\to \frac{f(u)}{|x|^{\delta}}$ for almost all $x\in\mathbb{R}^N$ as $n\to\infty$. Then
\begin{align*}
    \Bigl(\mathcal{I}_\mu\ast \frac{F(u_n)}{|x|^{\delta}}\Bigr)\frac{f(u_n)}{|x|^{\delta}} \to  \Bigl(\mathcal{I}_\mu\ast \frac{F(u)}{|x|^{\delta}}\Bigr)\frac{f(u)}{|x|^{\delta}}\quad \mbox{a.e.\ in}\ \mathbb{R}^N,\ \  \mbox{as}\ n\to\infty.
\end{align*}
By \eqref{hip:F2} and H\"{o}lder inequality, we have that for any $\varphi\in W^{s,p}_{V, \theta}(\mathbb{R}^N)$,
\begin{align*}
&\lefteqn{\int_{\mathbb{R}^N}\Bigl|\Bigl(\mathcal{I}_\mu\ast \frac{F(u_n)}{|x|^{\delta}}\Bigr)\frac{f(u_n)}{|x|^{\delta}} \varphi(x)\Bigr|\dd x}\\
&\leqslant c_0K\int_{\mathbb{R}^N}| \Bigl(\frac{|u_n|^{q_1-1}}{|x|^{\delta}}+\frac{|u_n|^{q_2-1}}{|x|^{\delta}}\Bigr) \varphi(x)|\dd x\\
&\leqslant c_0 K \Bigl[\Bigl(\int_{\mathbb{R}^N} \Bigl(\frac{|u_n|^{q_1-1}}{|x|^{\delta(q_1-1)/q_1}}\Bigr)^{\frac{q_1}{q_1-1}}\dd x\Bigr)^{\frac{q_1-1}{q_1}}\Bigl(\int_{\mathbb{R}^N}\Bigl(\frac{|\varphi(x)|}{|x|^{\delta/q_1}}\Bigr)^{q_1}\dd x\Bigr)^{\frac{1}{q_1}}\\
&\quad +\Bigl(\int_{\mathbb{R}^N} \Bigl(\frac{|u_n|^{q_2-1}}{|x|^{\delta(q_2-1)/q_2}}\Bigr)^{\frac{q_2}{q_2-1}}\dd x\Bigr)^{\frac{q_2-1}{q_2}}\Bigl(\int_{\mathbb{R}^N}\Bigl(\frac{|\varphi(x)|}{|x|^{\delta/q_2}}\Bigr)^{q_2}\dd x\Bigr)^{\frac{1}{q_2}}\Bigr]\\
&=c_0K  \Bigl(\|u_n\|_{{L^{q_1}(\mathbb{R}^N, |x|^{-\delta})}}^{q_1-1}\|\varphi\|_{L^{q_1}(\mathbb{R}^N, |x|^{-\delta})} + \|u_n\|_{{L^{q_2}(\mathbb{R}^N, |x|^{-\delta})}}^{q_2-1}\|\varphi\|_{{L^{q_2}(\mathbb{R}^N, |x|^{-\delta})}} \Bigr) \\
&\leqslant c_0K  \Big(C_{q_1}\|\ell(x)\|_{{L^{q_1}(\mathbb{R}^N, |x|^{-\delta})}}^{q_1-1}
+C_{q_2}\|\ell(x)\|_{{L^{q_2}(\mathbb{R}^N, |x|^{-\delta})}}^{q_2-1}\Big)\|\varphi\|_W.
\end{align*}
Then by Lebesgue dominated convergence theorem, we obtain
\begin{align*}
&\lefteqn{\|\Psi'(u_n)-\Psi'(u)\|_{\big(W_V^{s,p, \theta}(\mathbb{R}^N)\big)'}}\\
&=\sup\limits_{\|\varphi\|_{W_V^{s,p, \theta}(\mathbb{R}^N)}=1}|\langle \Psi'(u_n)-\Psi'(u), \varphi\rangle|\\
&=\sup\limits_{\|\varphi\|_{W_V^{s,p, \theta}(\mathbb{R}^N)}=1}\int_{\mathbb{R}^N}\Bigl|\Bigl(\mathcal{I}_\mu\ast \frac{F(u_n)}{|x|^{\delta}}\Bigr)\frac{f(u_n)}{|x|^{\delta}} \varphi(x)-\Bigl(\mathcal{I}_\mu\ast \frac{F(u)}{|x|^{\delta}}\Bigr)\frac{f(u)}{|x|^{\delta}} \varphi(x)\Bigr|\dd x\\
& \to 0\quad \mbox{as}\ n\to\infty.
\end{align*}
Therefore, we get that $\Psi'(u_n)\to\Psi'(u)$ in $\big(W_{V, \theta}^{s,p}(\mathbb{R}^N)\big)'$ as $n\to\infty$. This completes the proof.
\end{proof}

\section{The geometry of the mountain pass theorem}

In this section, we will prove our main result. First, we introduce the following definition.

\begin{definition}
\label{cerami}
For $c\in\mathbb{R}$, we say that $I$ satisfies the $(C)_c$ condition if for any sequence $\{u_n\}\subset W^{s,p}_{V, \theta}(\mathbb{R}^N)$ with
\begin{align*}
I(u_n)\to c,\quad \|I'(u_n)\|(1+\|u_n\|_W)\to 0,    
\end{align*}
there is a subsequence $\{u_n\}$ such that $\{u_n\}$ converges strongly in $W^{s,p}_{V, \theta}(\mathbb{R}^N)$.
\end{definition}

We will use the following  mountain pass theorem to prove  our result.
\begin{lemma}[Theorem 1 in \cite{costa1995nontrivial}]\label{cmainth}
Let $E$ be a real Banach space, $I\in C^1(E,\mathbb{R})$
satisfies the $(C)_c$ condition for any $c\in\mathbb{R}$, and

(i) There are constants $\rho,\alpha>0$ such that $I|_{\partial B_\rho}\geqslant\alpha$.

(ii) There is an $e\in E\backslash B_\rho$ such that $I(e)\leqslant 0$.\\
\noindent Then,
$$
c=\inf\limits_{\gamma\in\Gamma}\max\limits_{0\leqslant t\leqslant1}I(\gamma(t))\geqslant \alpha
$$
is a critical value of $I$, where
$$
\Gamma=\{\gamma\in C([0,1],E):\gamma(0)=0,\gamma(1)=e\}.
$$
\end{lemma}

We first show that the energy functional $I$ satisfies the geometric structure.

\begin{lemma}\label{mage}
Assume that \eqref{hip:V}, \eqref{hip:m1}--\eqref{hip:m2} and \eqref{hip:F1}--\eqref{hip:F3} hold. Then

(i) There exists $\alpha,\rho>0$ such that $I(u)\geqslant\alpha$ for all $u\in W^{s,p}_{V, \theta}(\mathbb{R}^N)$ with $\|u\|_W=\rho$.

(ii) $I(u)$ is unbounded from below on $W^{s,p}_{V, \theta}(\mathbb{R}^N)$.
\end{lemma}

\begin{proof}
$(i)$ From Lemma \ref{lema:2.3chen} and \eqref{hip:m1}--\eqref{hip:m2}, \eqref{hip:F2}, we have
\begin{align*}
 I(u)
&=\frac{1}{p}M(\|u\|^p_{W})
-\dfrac{1}{p_{s}^{\ast}(\beta,\theta)}
\int_{\mathbb{R}^{N}} 
\dfrac{|u|^{p_{s}^{\ast}(\beta,\theta)}}{|x|^{\beta}} \dd{x}\\
& \quad
-\dfrac{\lambda}{2 }
\int_{\mathbb{R}^{N}}
\int_{\mathbb{R}^{N}}
\dfrac{F_{\delta, \theta, \mu}(u(x))F_{\delta, \theta, \mu}(u(y))}%
       {|x|^{\delta}
        |x-y|^{\mu}
        |y|^{\delta}}
        \dd{x}\dd{y}\\
& \geqslant \frac{1}{p\xi}m(\|u\|_W^p)\|u\|_W^p- \|u\|_{W}^{p_s^*(\beta, \theta)}- \frac{\lambda c_0K }{2}\int_{\mathbb{R}^{N}}\Big(\frac{|u|^{q_1}}{q_1}+\frac{|u|^{q_2}}{q_2}\Big)\,\dd x\\
& \geqslant \left[\frac{m_0}{p\xi} - \|u\|_W^{p_s^*(\beta, \theta)-p}- \frac{\lambda c_0K }{2} \Big(C_{q_1}^{q_1}\|u\|_W^{q_1-p}+C_{q_2}^{q_2}\|u\|_W^{q_2-p}\Big)\right]\|u\|_W^p.
\end{align*}
Since $q_2\geqslant q_1>p$ and $p_s^*(\beta, \theta)>p$, the claim follows if we choose $\rho$ small enough.

$(ii)$ Rewriting the inequality of \eqref{hip:m2} in the form of $m(t)/M(t) \leqslant \xi/t$, after integration, we deduce that there is a constant $C \in \mathbb{R}_+$ such that
\begin{align}\label{eq:3.1chen}
{M}(t)\leqslant Ct^\xi\quad \mbox{for\ all}\  t\geqslant 1.
\end{align}
By the assumption \eqref{hip:F3}, we can take that $t_0$ such that $F(t_0)\neq 0$, we find
\begin{align*}
\int_{\mathbb{R}^N}\Bigl(\mathcal{I}_\mu\ast \frac{F(t_0\chi_{B_1})}{|x|^{\delta}}\Bigr)\frac{F(t_0\chi_{B_1})}{|x|^{\delta}}\dd x=F(t_0)^2\int_{B_1}\int_{B_1}\frac{1}{|x|^{\delta}|x-y|^{\mu}|y|^{\delta}}\dd x \dd y>0,    
\end{align*}
where $B_r$ denotes the open ball centered at the origin with radius $r$ and $\chi_{B_1}$ denotes the standard indicator function of set $B_1$. By the density theorem, there will be
$v_0\in W^{s,p}_{V, \theta}(\mathbb{R}^N)$ with
\begin{align*}
\int_{\mathbb{R}^N}\Bigl(\mathcal{I}_\mu\ast \frac{F(v_0)}{|x|^{\delta}}\Bigr)\frac{F(v_0)}{|x|^{\delta}}\dd x>0.    
\end{align*}
Define the function $v_t(x)=v_0(\frac{x}{t})$, then, using the change of variables $x/t=\Bar{x}$ and $y/t = \Bar{y}$, we have
\begin{align*}
I(v_t)
&=\frac{1}{p}{M}(\|v_t\|_W^p)- \frac{1}{p_s^*(\beta, \theta)} \int_{\mathbb{R}^N} \frac{|v_t|^{p_s^*(\beta, \theta)}}{|x|^{\beta}} \dd x-  \frac{\lambda }{2}\iint_{\mathbb{R}^{2N}}\frac{ F(v_t(x))F(v_t(y))}{|x|^{\delta}|x-y|^\mu|y|^{\delta}}\dd x\dd y\\
& \leqslant \frac{1}{p}C\|v_t\|_W^{p\xi} - \frac{1}{p_s^*(\beta, \theta)} \int_{\mathbb{R}^N} \frac{|v_t|^{p_s^*(\beta, \theta)}}{|x|^{\beta}} \dd x -  \frac{\lambda }{2}\iint_{\mathbb{R}^{2N}}\frac{ F(v_t(x))F(v_t(y))}{|x|^{\delta}|x-y|^\mu|y|^{\delta}}\dd \Bar{x}\dd \Bar{y}\\
&=\frac{1}{p}C\left[t^{N-ps-\theta}\iint_{\mathbb{R}^{2N}}\frac{|v_0(\Bar{x})- v_0(\Bar{y})|^p}{|x|^{\theta_1}|x-y|^{N+sp}|y|^{\theta_2}}\dd \Bar{x}\dd \Bar{y}+t^{N-sp-\theta}\int_{\mathbb{R}^N}\frac{V(t\Bar{x})|v_0|^p}{|\Bar{x}|^{sp+\theta}} \dd \Bar{x}\right]^\xi\\
& \quad - \frac{t^{N-\beta}}{p_s^*(\beta, \theta)} \int_{\mathbb{R}^N} \frac{|v_0|^{p_s^*(\beta, \theta)}}{|\Bar{x}|^{\beta}} \dd \Bar{x} -  t^{2(N-\delta- \mu/2)}\frac{\lambda }{2}\iint_{\mathbb{R}^{2N}}\frac{ F(v_0(x))F(v_0(y))}{|x|^{\delta}|x-y|^\mu|y|^{\delta}}\dd \Bar{x}\dd \Bar{y},
\end{align*}
for sufficiently large $t$. Therefore, we have that $I(v_t)\to-\infty$ as $t\to\infty$ since $1\leqslant \xi<\frac{2(N-\delta -\mu/2)}{N}$ gives that $2(N-\delta-\mu/2)>N\xi>(N-ps-\theta)\xi$. Furthermore, since $\beta < sp+ \theta $, then $N-\beta>N -sp -\theta$. Hence we obtain that the functional $I$ is unbounded from below.
\end{proof}

\section{The compactness of the Cerami sequences}
Next, we prove the important result that the Cerami sequences for the energy functional are bounded.

\begin{lemma}\label{lema:3.4chen}
Assume that \eqref{hip:V}, \eqref{hip:m1}--\eqref{hip:m2} and \eqref{hip:F1}--\eqref{hip:F4} hold. Then $(C)_c-$sequence of $I$ is bounded for any $\lambda>0$.
\end{lemma}

\begin{proof}
Suppose that  $\{u_n\}\subset W^{s,p}_{V, \theta}(\mathbb{R}^N)$ is a $(C)_c$ sequence for $I(u)$, that is, $I(u_n)\to c$ and $\|I'(u_n)\|_W(1+\|u_n\|_W)\to 0$ as $n\to+\infty$; then
\begin{align}
\label{eq:3.2chen}
c=I(u_n)+o(1),\quad \text{and}\quad \langle I'(u_n),u_n\rangle=o(1)
\end{align}
where $o(1)\to0$ as $n\to+\infty$. We now prove that $\{u_n\}$ is bounded in $W^{s,p}_{V, \theta}(\mathbb{R}^N)$. We argue by contradiction. Suppose that the sequence $\{u_n\}$ is unbounded in $W^{s,p}_{V, \theta}(\mathbb{R}^N)$, then we may assume that
\begin{align}
\label{eq:3.3chen}
\|u_n\|_W\to\infty \qquad (n\to+\infty).
\end{align}
Let $\omega_n(x)=\frac{u_n}{\|u_n\|_W}$, then $\omega_n\in W^{s,p}_{V, \theta}(\mathbb{R}^N)$ with $\|\omega_n\|_W=1$. Hence, up to a subsequence, still denoted by itself, there exists a function $\omega\in W^{s,p, \theta}_V(\mathbb{R}^N)$ such that
\begin{align}
\label{eq:3.4chen}
\omega_n(x)\to\omega(x)\quad \mbox{a.e.\ in}\ \mathbb{R}^N,\qquad \mbox{and}\ \
\omega_n(x)\to\omega(x)\quad \mbox{a.e.\ in}\ L^r(\mathbb{R}^N)
\end{align}
as $n\to\infty$,  for $p\leqslant r<\frac{Np}{N-ps-\theta}$.

Let $\Omega_1=\{x\in\mathbb{R}^N: \omega(x)\neq0\}$; then
\[
\lim\limits_{n\to\infty}\omega_n(x)=\lim\limits_{n\to\infty}\frac{u_n(x)}{\|u_n\|_W}=\omega(x)\neq 0\ \ \mbox{in}\ \Omega_1,
\]
and (\ref{eq:3.3chen}) implies that
\begin{align}
\label{eq:3.5chen}
|u_n|\to\infty \quad \mbox{a.e.\ in}\ \Omega_1.
\end{align}
So, from the assumption \eqref{hip:F3} and Lemma \ref{lema:2.3chen}, we have
\begin{align}
\label{eq:3.6chen}
 \lim\limits_{n\to\infty}\frac{\Bigl(\mathcal{I}_\mu\ast \frac{F(u_n(x))}{|x|^{\delta}}\Bigr)\frac{F(u_n(x))}{|x|^{\delta}}}{|u_n(x)|^{p\xi}}|\omega_n(x)|^{p\xi}=\infty,\ \ \mbox{for\ a.e.}\ x\in\Omega_1.
\end{align}

Hence, there is a constant $C$ such that
\begin{align}
\label{eq:3.7chen}
\frac{\Bigl(\mathcal{I}_\mu\ast \frac{F(u_n)}{|x|^{\delta}}\Bigr)\frac{F(u_n(x))}{|x|^{\delta}}}{|u_n(x)|^{p\xi}}|\omega_n(x)|^{p\xi}-\frac{ C}{\|u_n\|_W^{p\xi}}\geqslant 0.
\end{align}
By (\ref{eq:3.2chen}) we have that
\begin{align}
\label{eq:3.8chen}
c&=I(u_n)+o(1) \nonumber\\
&=\dfrac{1}{p}
M(\|u_n\|^p_{{W}})
-\dfrac{1}{p_{s}^{\ast}(\beta,\theta)}
\int_{\mathbb{R}^{N}} 
\dfrac{|u_n|^{p_{s}^{\ast}(\beta,\theta)}}{|x|^{\beta}} \dd{x} \nonumber\\
& \quad
-\dfrac{\lambda}{2 }
\iint_{\mathbb{R}^{2N}}
\dfrac{F_{\delta, \theta, \mu}(u_n(x))F_{\delta, \theta, \mu}(u_n(y))}%
       {|x|^{\delta}
        |x-y|^{\mu}
        |y|^{\delta}}
        \dd{x}\dd{y}+o(1).
\end{align}
Using this estimate, 
the embedding $W_{V, \theta}^{s,p}(\mathbb{R}^N) \hookrightarrow L^{{p_{s}^{\ast}(\beta,\theta)}}(\mathbb{R}^N, |y|^{-\beta})$
together with hypotheses \eqref{hip:m1}--\eqref{hip:m2}, 
we find
\begin{align}
\label{eq:3.9chen}
{} & \max\Bigl\{\frac{1}{2}, \frac{1}{\lambda p_s^{\ast}(\beta, \theta)}\Bigr\} \Bigl(\int_{\mathbb{R}^{N}}\Bigl(\mathcal{I}_\mu\ast \frac{F(u_n)}{|x|^{\delta}}\Bigr)\frac{F(u_n)}{|x|^{\delta}} \dd x+ \int_{\mathbb{R}^{N}} 
\dfrac{|u_n|^{p_{s}^{\ast}(\beta,\theta)}}{|x|^{\beta}} \dd{x} \Bigr)\nonumber \\
& \geqslant \frac{1}{2}\int_{\mathbb{R}^{N}}\Bigl(\mathcal{I}_\mu\ast \frac{F(u_n)}{|x|^{\delta}}\Bigr)\frac{F(u_n)}{|x|^{\delta}}\dd x + \dfrac{1}{\lambda p_{s}^{\ast}(\beta,\theta)}
\int_{\mathbb{R}^{N}} 
\dfrac{|u_n|^{p_{s}^{\ast}(\beta,\theta)}}{|x|^{\beta}} \dd{x}\nonumber\\
& \geqslant \frac{1}{\xi p\lambda} m(\|u_n\|_W^{p})\|u_n\|_W^{p}-  \frac{c}{\lambda}+\frac{o(1)}{\lambda}\nonumber\\
& \geqslant \frac{m_0}{\xi p\lambda} \|u_n\|_W^{p} -  \frac{c}{\lambda}+\frac{o(1)}{\lambda}\nonumber\\
 & \to\infty,\quad\mbox{as}\ \ n\to\infty.
\end{align}
We claim that $\mbox{meas}(\Omega_1)=0$. Indeed, if $\mbox{meas}(\Omega_1)\neq 0$.
From (\ref{eq:3.6chen}) and (\ref{eq:3.7chen}), we have
\begin{align}
\label{eq:3.10chen}
+\infty&=
\int_{\Omega_1}\liminf\limits_{n\to\infty}\frac{\Bigl(\mathcal{I}_\mu\ast \frac{F(u_n(x))}{|x|^{\delta}}\Bigr)\frac{F(u_n(x))}{|x|^{\delta}}}{|u_n(x)|^{p\xi}}|\omega_n(x)|^{p\xi}\dd x \nonumber\\
& \quad + \int_{\Omega_1}\liminf\limits_{n\to\infty} \frac{|u_n|^{p_s^{\ast}(\beta, \theta)}}{|x|^{\beta}|u_n(x)|^{p\xi}}|\omega_n(x)|^{p\xi} \dd x 
 - \int_{\Omega_1}\limsup\limits_{n\to\infty}\frac{ C }{ \|u_n\|_W^{p\xi} }\dd x\nonumber\\
&\leqslant \int_{\Omega_1}\liminf\limits_{n\to\infty}\left(\frac{\Bigl(\mathcal{I}_\mu\ast \frac{F(u_n(x))}{|x|^{\delta}}\Bigr)\frac{F(u_n(x))}{|x|^{\delta}}}{|u_n(x)|^{p\xi}}|\omega_n(x)|^{p\xi}+ \frac{|u_n|^{p_s^{\ast}(\beta, \theta)}}{|x|^{\beta}|u_n(x)|^{p\xi}}|\omega_n(x)|^{p\xi}
-  \frac{ C }{ \|u_n\|_W^{p\xi} }\right)\dd x\nonumber\\
\intertext{and by Fatou's lemma,}
& \leqslant \liminf\limits_{n\to\infty}\int_{\Omega_1}\left(\frac{\Bigl(\mathcal{I}_\mu\ast \frac{F(u_n(x))}{|x|^{\delta}}\Bigr)\frac{F(u_n(x))}{|x|^{\delta}}}{|u_n(x)|^{p\xi}}|\omega_n(x)|^{p\xi} + \frac{|u_n|^{p_s^{\ast}(\beta, \theta)}}{|x|^{\beta}|u_n(x)|^{p\xi}}|\omega_n(x)|^{p\xi}
-  \frac{ C }{ \|u_n\|_W^{p\xi} }\right)\dd x\nonumber\\
&=\liminf\limits_{n\to\infty}\int_{\Omega_1}\left(\frac{\Bigl(\mathcal{I}_\mu\ast \frac{F(u_n)}{|x|^{\delta}}\Bigr)\frac{F(u_n)}{|x|^{\delta}} }{  \|u_n\|_W^{p\xi} }
+ \frac{|u_n|^{p_s^{\ast}(\beta, \theta)}}{|x|^{\beta}\|u_n\|_W^{p\xi}}- \frac{ C }{ \|u_n\|_W^{p\xi} }\right)\dd x\nonumber\\
\intertext{and by \eqref{eq:3.1chen},}
& \leqslant \liminf\limits_{n\to\infty}\int_{\Omega_1} \Bigl(\frac{C\Bigl(\mathcal{I}_\mu\ast \frac{F(u_n)}{|x|^{\delta}}\Bigr)\frac{F(u_n)}{|x|^{\delta}} }{{M}(\|u_n\|_W^{p}) } + \frac{C|u_n|^{p_s^{\ast}(\beta, \theta)}}{|x|^{\beta}{M}(\|u_n\|_W^{p}) }\Bigr) \dd x
- \liminf\limits_{n\to\infty} \int_{\Omega_1}\frac{C }{ \|u_n\|_W^{p\theta} }\dd x\nonumber\\
& \leqslant \liminf\limits_{n\to\infty}\int_{\mathbb{R}^N} \Bigl(\frac{C\Bigl(\mathcal{I}_\mu\ast \frac{F(u_n)}{|x|^{\delta}}\Bigr)\frac{F(u_n)}{|x|^{\delta}} }{{M}(\|u_n\|_W^{p}) } + \frac{C\frac{|u_n|^{p_s^{\ast}(\beta, \theta)}}{|x|^{\beta}}}{{M}(\|u_n\|_W^{p}) }\Bigr) \dd x \nonumber \\
&=\frac{ C}{p}\liminf\limits_{n\to\infty}\int_{\mathbb{R}^N} \frac{\Bigl(\mathcal{I}_\mu\ast \frac{F(u_n)}{|x|^{\delta}}\Bigr)\frac{F(u_n)}{|x|^{\delta}} + \frac{|u_n|^{p_s^{\ast}(\beta, \theta)}}{|x|^{\beta}} }{ \frac{1}{p}{M}(\|u_n\|_W^{p}) }\dd x\nonumber\\
\intertext{and by \eqref{eq:3.8chen},}
&=\frac{ C}{p}\liminf\limits_{n\to\infty} \frac{\int_{\mathbb{R}^N}\Bigl(\Bigl(\mathcal{I}_\mu\ast \frac{F(u_n)}{|x|^{\delta}}\Bigr)\frac{F(u_n)}{|x|^{\delta}} + \frac{|u_n|^{p_s^{\ast}(\beta, \theta)}}{|x|^{\beta}}\Bigr)\dd x}{ \frac{\lambda}{2}\int_{\mathbb{R}^{N}}\Bigl(\mathcal{I}_\mu\ast \frac{F(u_n)}{|x|^{\delta}}\Bigr)\frac{F(u_n)}{|x|^{\delta}}\dd x+\frac{1}{p_s^{\ast}(\beta, \theta)}\int_{\mathbb{R}^N}\frac{|u_n|^{p_s^{\ast}(\beta, \theta)}}{|x|^{\beta}}\dd x+ c-o(1) } \nonumber \\
& \leqslant \frac{ C}{p}\liminf\limits_{n\to\infty} \frac{\int_{\mathbb{R}^N}\Bigl(\Bigl(\mathcal{I}_\mu\ast \frac{F(u_n)}{|x|^{\delta}}\Bigr)\frac{F(u_n)}{|x|^{\delta}} + \frac{|u_n|^{p_s^{\ast}(\beta, \theta)}}{|x|^{\beta}}\Bigr)\dd x}{ \max\Bigl\{\frac{\lambda}{2}, \frac{1}{p_s^{\ast}(\beta, \theta)}\Bigr\}\Bigl(\int_{\mathbb{R}^{N}}\Bigl(\mathcal{I}_\mu\ast \frac{F(u_n)}{|x|^{\delta}}\Bigr)\frac{F(u_n)}{|x|^{\delta}}\dd x+\int_{\mathbb{R}^N}\frac{|u_n|^{p_s^{\ast}(\beta, \theta)}}{|x|^{\beta}}\dd x\Bigr)+ c-o(1) }.
\end{align}
So, by (\ref{eq:3.9chen}) and (\ref{eq:3.10chen}) we get the contradiction
$+\infty\leqslant C/[p\max\{\lambda/2, 1/p_s^{\ast}(\beta, \theta)]$.
This shows that $\mbox{meas}(\Omega_1)=0$. Hence, $\omega(x)=0$ for almost all $x\in\mathbb{R}^N$ and the convergence in (\ref{eq:3.4chen}) means that
\begin{align}
\label{eq:3.11chen}
\omega_n(x)\to 0\quad \mbox{a.e.\ in}\ \mathbb{R}^N,\quad \mbox{and}\ \
\omega_n(x)\to0\quad \mbox{a.e.\ in}\ L^r(\mathbb{R}^N)\ \ \ \mbox{as}\ n\to\infty,
\end{align}
for $p\leqslant r< \frac{Np}{N-ps-\theta}$.

Using (\ref{eq:3.2chen}), \eqref{hip:m2}, $p_s^{\ast}(\beta, \theta)>p$ and $\xi \geqslant 1$, we get
\begin{align}
\label{eq:3.12chen}
c+1\geqslant&I(u_n)-\frac{1}{p\xi}\langle I'(u_n),u_n\rangle\nonumber\\
& = \frac{1}{p}{M}(\|u_n\|_W^p)-   \frac{1}{p\xi}m(\|u_n\|_W^p)\|u_n\|_W^p\nonumber\\
& \quad + \Bigl( \frac{1}{p\xi} -\frac{1}{p_s^{\ast}(\beta, \theta)} \Bigr)\int_{\mathbb{R}^N} \frac{|u_n|^{p_s^{\ast}(\beta, \theta)}}{|x|^{\beta}} \dd x\nonumber\\
&\quad +\lambda \int_{\mathbb{R}^{N}}\Bigl(\mathcal{I}_\mu\ast \frac{F(u_n)}{|x|^{\delta}}\Bigr)\left(\frac{1}{p\xi}\frac{f(u_n)}{|x|^{\delta}}u_n-\frac{1}{2}\frac{F(u_n)}{|x|^{\delta}}\right)\dd x\nonumber\\
& \geqslant \lambda \int_{\mathbb{R}^{N}}(\mathcal{I}_\mu\ast F(u_n))\mathscr{F}(u_n)\dd x,
\end{align}
for $n$ large enough.

For $a,b\geqslant 0$, let us define 
\begin{align*}
    \Omega_n^{\ast}(a,b)
    &\coloneqq\Big\{x\in\mathbb{R}^N:a\leqslant \frac{|u_n(x)|}{|x|^{\beta/(p_s^{\ast}(\beta, \theta)-p)}}\leqslant b\Big\} \\
    \Omega_n^i(a,b)
    &\coloneqq\Big\{x\in\mathbb{R}^N:a\leqslant \frac{|u_n(x)|}{|x|^{\delta/(q_i-p)}}\leqslant b\Big\}\quad (i\in\{1,2\}).
\end{align*}
From \eqref{hip:m1} and \eqref{hip:m2}, we have that
\begin{align}
\label{eq:3.13chen}
{M}(\|u_n\|_W^p)\geqslant \frac{1}{\xi}m(\|u_n\|_W^p) \|u_n\|_W^p \geqslant\frac{m_0}{\xi}  \|u_n\|_W^p.
\end{align}
This inequality, together with (\ref{eq:3.3chen}) and (\ref{eq:3.8chen}) yields that
\begin{align}\label{eq:3.14chen}
0<\frac{\lambda}{2p}
&\leqslant 
\limsup\limits_{n\to\infty}\frac{\int_{\mathbb{R}^{N}}\Bigl(\mathcal{I}_\mu\ast \frac{F(u_n)}{|x|^{\delta}}\Bigr)\frac{F(u_n)}{|x|^{\delta}} \dd x } {{M}(\|u_n\|_W^p)}\nonumber\\
 & = \limsup\limits_{n\to\infty}\int_{\mathbb{R}^{N}}\frac{\Bigl(\mathcal{I}_\mu\ast \frac{F(u_n)}{|x|^{\delta}}\Bigr)\frac{F(u_n)}{|x|^{\delta}}} {{M}(\|u_n\|_W^p)}\dd x\nonumber\\
&=\limsup\limits_{n\to\infty}\left(\int_{\Omega_n^1\cap \Omega_n^2\cap \Omega_n^{\ast}(0,r_0)}+ \int_{\mathbb{R}^N \backslash \Omega_n^1\cap \Omega_n^2\cap \Omega_n^{\ast}(0,r_0)}\right) \frac{\Bigl(\mathcal{I}_\mu\ast \frac{F(u_n)}{|x|^{\delta}}\Bigr)\frac{F(u_n)}{|x|^{\delta}} } {{M}(\|u_n\|_W^p)}\dd x.
\end{align}

To simplify the notation, from now on we write $\Omega_n^1\cap \Omega_n^2\cap \Omega_n^{\ast}(0,r_0) = \Omega(0,r_0)$. On the one hand, by Lemma \ref{lema:2.3chen}, (\ref{eq:3.13chen}),
\eqref{hip:F2}, and~(\ref{eq:3.11chen}), we obtain
\begin{align}\label{eq:3.15chen}
\lefteqn{\int_{\Omega(0,r_0)} \frac{\Bigl(\mathcal{I}_\mu\ast \frac{F(u_n)}{|x|^{\delta}}\Bigr)\frac{F(u_n)}{|x|^{\delta}} } {{M}(\|u_n\|_W^p)}\dd x} \nonumber \\
& \leqslant \int_{\Omega(0,r_0)} \frac{K\frac{|F(u_n)|}{|x|^{\delta}} } {{M}(\|u_n\|_W^p)}\dd x \nonumber \\
&\leqslant  \frac{K\xi}{m_0}\int_{\Omega(0,r_0)} \frac{ \frac{|F(u_n)|}{|x|^{\delta}} } { \|u_n\|_W^p }\dd x\nonumber\\
&\leqslant \frac{c_0K\xi}{m_0}\int_{\Omega(0,r_0)}   \frac{1}{|x|^{\delta}}\left(\frac{ |u_n|^{q_1-p}} {q_1 }|\omega_n|^p+\frac{ |u_n|^{q_2-p}} {q_2 }|\omega_n|^p\right)\dd x + \frac{\xi}{m_0} \int_{\Omega(0,r_0)}  \frac{|u_n|^{p_s^{\ast}(\beta, \theta)-p}}{|x|^{\beta}}|\omega_n|^p\dd x \nonumber\\
&\leqslant  \frac{c_0K\xi}{m_0} \left(\frac{ r_0^{q_1-p}} {q_1 } +\frac{ r_0^{q_2-p}} {q_2 }\right)\int_{\Omega(0,r_0)} |\omega_n|^p\dd x + \frac{\xi}{m_0}r_0^{p_s^{\ast}(\beta, \theta)-p} \int_{\Omega(0,r_0)} |\omega_n|^p\dd x \nonumber \\
&\to 0,\ \ \mbox{as}\ n\to\infty.
\end{align}
On the other hand, using H\"{o}lder inequality,  (\ref{eq:3.11chen}), (\ref{eq:3.12chen}) and \eqref{hip:F4}, we find
\begin{align}\label{eq:3.16chen}
\lefteqn{\int_{\mathbb{R}^N \backslash \Omega(0,r_0)}  \frac{\Bigl|\mathcal{I}_\mu\ast \frac{F(u_n)}{|x|^{\delta}}\Bigr|\frac{F(u_n)}{|x|^{\delta}}} {{M}(\|u_n\|_W^p)}\dd x}\nonumber\\
&\leqslant\frac{\xi}{m_0}c_1^{\frac{1}{\kappa}}K^{\frac{\kappa-1}{\kappa}}\left(\int_{\mathbb{R}^N \backslash \Omega(0,r_0)}  \Bigl|\mathcal{I}_\mu\ast \frac{F(u_n)}{|x|^{\delta}}\Bigr| \mathscr{F}(u_n) \dd x\right)^{\frac{1}{\kappa}}
\left(\int_{\mathbb{R}^N \backslash \Omega(0,r_0)}   |\omega_n(x)|^{\frac{\kappa p}{\kappa-1}}\dd x\right)^{\frac{\kappa-1}{\kappa}}\nonumber\\
& \leqslant\frac{\xi}{m_0}c_1^{\frac{1}{\kappa}}K^{\frac{\kappa-1}{\kappa}}\left(\frac{c+1}{\lambda}\right)^{\frac{1}{\kappa}}
\left(\int_{\mathbb{R}^N \backslash \Omega(0,r_0)}   |\omega_n(x)|^{\frac{\kappa p}{\kappa-1}}\dd x\right)^{\frac{\kappa-1}{\kappa}}\to0,\ \   \mbox{as}\ n\to\infty.
\end{align}
Here we used the fact that $\frac{\kappa p}{\kappa-1}\in (p,\frac{p(N-\beta)}{N-ps-\theta})$ if $\kappa>\frac{N-\beta}{ps+\theta-\beta}$.
Thus, we get a contradiction from (\ref{eq:3.14chen})-(\ref{eq:3.16chen}). The proof is complete. 
\end{proof}

\begin{lemma}\label{mage2}
Assume that \eqref{hip:V}, \eqref{hip:m1}--\eqref{hip:m2} and \eqref{hip:F1}--\eqref{hip:F4} hold. Then the functional $I$ satisfies $(C)_c-$condition for any $\lambda>0$.
\end{lemma}

\begin{proof}
Suppose that  $\{u_n\}\subset W^{s,p}_{V, \theta}(\mathbb{R}^N)$ is a $(C)_c$ sequence for $I$, from Lemma \ref{lema:3.4chen}, we have that $\{u_n\}$ is bounded in $W^{s,p}_{V, \theta}(\mathbb{R}^N)$, then if necessary to a subsequence, we have
\begin{align}
\label{eq:3.17chen}
&u_n\rightharpoonup u \ \ \mbox{in}\ W^{s,p}_{V, \theta}(\mathbb{R}^N),\quad u_n\to u\ \ \mbox{a.e.\ in}\ \mathbb{R}^N,\nonumber\\
&u_n\to u\ \ \mbox{in} \ L^{q_1}(\mathbb{R}^N,|x|^{-\delta})\cap L^{q_2}(\mathbb{R}^N, |x|^{-\delta}),\\%\ \ \mbox{for}\ p\leqslant r<p^\ast_s,\\
&\frac{|u_n|}{|x|^{\delta}}\leqslant \ell(x)\ \ \mbox{a.e.\ in}\ \mathbb{R}^N,\ \ \mbox{for\ some}\ \ell(x)\in L^{q_1}(\mathbb{R}^N, |x|^{-\delta})\cap L^{q_2}(\mathbb{R}^N, |x|^{-\delta}).\nonumber
\end{align}
For simplicity, let $\varphi\in W_{V, \theta}^{s,p}(\mathbb{R}^N)$ be fixed and denote by $B_\varphi$ the linear functional on
$W_{V, \theta}^{s,p}(\mathbb{R}^N)$ defined by
\begin{align*}
B_\varphi(v)&=\iint_{\mathbb{R}^{2N}}\frac{|\varphi(x)-\varphi(y)|^{p-2}(\varphi(x)-\varphi(y))
}{|x|^{\theta_1}|x-y|^{N+ps}|y|^{\theta_2}}(v(x)-v(y))\dd x\dd y.
\end{align*}
for all $v\in W_{V, \theta}^{s,p}(\mathbb{R}^N)$.
By H\"older inequality, we have
\begin{align*}
\lefteqn{\iint_{\mathbb{R}^{2N}}\frac{|\varphi(x)-\varphi(y)|^{p-2}(\varphi(x)-\varphi(y))
}{|x|^{\theta_1}|x-y|^{N+ps}|y|^{\theta_2}}(v(x)-v(y))\dd x\dd y} \\
& = \iint_{\mathbb{R}^{2N}}\frac{|\varphi(x)-\varphi(y)|^{p-1}
}{|x|^{\theta_1-\theta_1/p}|x-y|^{(N+ps)-(N+ps)/p}|y|^{\theta_2-\theta_2/p}}\frac{(v(x)-v(y))}{|x|^{\theta_1/p}|x-y|^{(N+ps)/p}|y|^{\theta_2/p}}\dd x\dd y \\
& \leqslant  \Bigl(\iint_{\mathbb{R}^{2N}}\Bigl(\frac{|\varphi(x)-\varphi(y)|^{p-1}
}{|x|^{\theta_1-\theta_1/p}|x-y|^{(N+ps)-(N+ps)/p}|y|^{\theta_2-\theta_2/p}}\Bigr)^{\frac{p}{p-1}} \dd x\dd y\Bigr)^{\frac{p-1}{p}}\\
& \quad  \Bigl(\iint_{\mathbb{R}^{2N}}\Bigl(\frac{(v(x)-v(y))}{|x|^{\theta_1/p}|x-y|^{(N+ps)/p}|y|^{\theta_2/p}}\Bigr)^{p} \dd x\dd y\Bigr)^{\frac{1}{p}}\\
& \leqslant [\varphi]_{s,p, \theta}^{p-1}[v]_{s,p, \theta} \leqslant 
\|\varphi\|_W^{p-1}\|v\|_W,
\end{align*}
for all $v\in W_{V, \theta}^{s,p}(\mathbb{R}^N)$.
Hence, from~\eqref{eq:3.17chen} we deduce that
\begin{align}
\label{eq:3.18chen}
\lim\limits_{n\to\infty}\Big(m(\|u_n\|_W^p)-m(\|u\|_W^p)\Big)B_u(u_n-u)=0,
\end{align}
because the sequence $\{m(\|u_n\|_W^p)-m(\|u\|_W^p)\}_{n\in\mathbb{N}} \subset \mathbb{R}$ is bounded.

Since $I'(u_n)\to0$ in $(W^{s,p}_{V, \theta}(\mathbb{R}^N))'$ and $u_n\rightharpoonup u$ in  $W^{s,p}_{V, \theta}(\mathbb{R}^N)$, we have
\begin{align*}
    \langle I'(u_n)-I'(u),u_n-u\rangle\to0
    \quad \text{as } n\to\infty.
\end{align*}
Following up, we get
\begin{align}
\label{eq:3.19chen}
o(1)&=\langle I'(u_n)-I'(u),u_n-u\rangle\nonumber\\
&= m(\|u_n\|_W^p)\|u_n-u\|_W^p - m(\|u\|_W^p)\|u_n-u\|_W^p\nonumber\\
&\quad- \int_{\mathbb{R}^N} \frac{|u_n|^{p^{\ast}_s(\beta,\theta)}u_n(u_n-u)}{|x|^{\beta}}\dd x + \int_{\mathbb{R}^N} \frac{|u|^{p^{\ast}_s(\beta,\theta)}u(u_n-u)}{|x|^{\beta}}\dd x \nonumber \\
&\quad - \lambda \int_{\mathbb{R}^N} \Bigl(\mathcal{I}_{\mu} \ast \frac{F(u_n)}{|x|^{\delta}}\Bigr)\frac{f(u_n)}{|x|^{\delta}}(u_n-u)\dd x  + \lambda \int_{\mathbb{R}^N} \Bigl(\mathcal{I}_{\mu} \ast \frac{F(u)}{|x|^{\delta}}\Bigr)\frac{f(u)}{|x|^{\delta}}(u_n-u)\dd x \nonumber \\
& = m(\|u_n\|_W^p)\Big(B_{u_n}(u_n-u)+\int_{\mathbb{R}^N}\frac{V(x)|u_n|^{p-2}u_n(u_n-u)}{|x|^{\alpha}}\dd x\Big)\nonumber\\
& \quad -m(\|u\|_W^p)\Big(B_{u}(u_n-u)+\int_{\mathbb{R}^N}\frac{V(x)|u|^{p-2}u(u_n-u)}{|x|^{\alpha}}\dd x\Big)\nonumber\\
& \quad +m(\|u_n\|_W^p)B_{u}(u_n-u) - m(\|u_n\|_W^p)B_{u}(u_n-u) \nonumber \\
& \quad +m(\|u_n\|_W^p)\frac{V(x)|u|^{p-2}u(u_n-u)}{|x|^{\alpha}}\dd x - m(\|u_n\|_W^p)\frac{V(x)|u|^{p-2}u(u_n-u)}{|x|^{\alpha}}\dd x \nonumber \\
& \quad - \int_{\mathbb{R}^N}\frac{(|u_n|^{p^{\ast}_s(\beta, \theta)-2}u_n- |u|^{p^{\ast}_s(\beta, \theta)-2}u)(u_n-u)}{|x|^{\alpha}} \dd x\nonumber \\
& \quad -\lambda\int_{\mathbb{R}^N}\Big[\Bigl(\mathcal{I}_\mu\ast \frac{F(u_n)}{|x|^{\delta}}\Bigr)\frac{f(u_n)}{|x|^{\delta}}-\Bigl(\mathcal{I}_\mu\ast \frac{F(u)}{|x|^{\delta}}\Bigr)\frac{f(u)}{|x|^{\delta}}\Big](u_n-u)\dd x\nonumber\\
& = m(\|u_n\|_W^p) \Big[B_{u_n}(u_n-u)- B_{u}(u_n-u)\Big] \nonumber\\
& \quad +\Big(m(\|u_n\|_W^p)-m(\|u\|_W^p)\Big) B_{u}(u_n-u)\nonumber\\
& \quad + m(\|u_n\|_W^p) \int_{\mathbb{R}^N}\frac{V(x)(|u_n|^{p-2}u_n-|u|^{p-2}u)(u_n-u)}{|x|^{\alpha}}\dd x\nonumber\\
& \quad +[m(\|u_n\|_W^p)-m(\|u\|_W^p)] \int_{\mathbb{R}^N}\frac{V(x)|u|^{p-2}u(u_n-u)}{|x|^{\alpha}}\dd x\nonumber\\
& \quad - \int_{\mathbb{R}^N}\frac{(|u_n|^{p^{\ast}_s(\beta, \theta)-2}u_n- |u|^{p^{\ast}_s(\beta, \theta)-2}u)(u_n-u)}{|x|^{\alpha}} \dd x\nonumber \\
& \quad -\lambda\int_{\mathbb{R}^N}\Big[\Bigl(\mathcal{I}_\mu\ast \frac{F(u_n)}{|x|^{\delta}}\Bigr)\frac{f(u_n)}{|x|^{\delta}}-\Bigl(\mathcal{I}_\mu\ast \frac{F(u)}{|x|^{\delta}}\Bigr)\frac{f(u)}{|x|^{\delta}}\Big](u_n-u)\dd x.
\end{align}
From Lemma \ref{lema:2.4chen}, we have
\begin{align}
\label{eq:3.20chen}
\int_{\mathbb{R}^N}\Big[\Bigl(\mathcal{I}_\mu\ast \frac{F(u_n)}{|x|^{\delta}}\Bigr)\frac{f(u_n)}{|x|^{\delta}}-\Bigl(\mathcal{I}_\mu\ast \frac{F(u)}{|x|^{\delta}}\Bigr)\frac{f(u)}{|x|^{\delta}}\Big](u_n-u)\dd x\to0,\quad \mbox{as}\ n\to\infty.
\end{align}
Moreover, using H\"{o}lder inequality, we have
\begin{align}
 \int_{\mathbb{R}^N}\frac{V(x)|u|^{p-2}u(u_n-u)}{|x|^{\alpha}}\dd x
& =  
\int_{\mathbb{R}^N}\frac{(V(x))^{(p-1)/p}|u|^{p-1}}{|x|^{\alpha(p-1)/p}}\frac{(V(x))^{1/p}|u_n-u|^{p}}{|x|^{\alpha/p}}\dd x \nonumber \\
& \leqslant \Bigl(\int_{\mathbb{R}^N}\frac{V(x)|u|^p}{|x|^{\alpha}}\dd x\Bigr)^{\frac{p-1}{p}}
\Bigl(\int_{\mathbb{R}^N}\frac{V(x)|u_n-u|^p}{|x|^{\alpha}}\dd x\Bigr)^{\frac{1}{p}}.
\end{align}
From inequality above and~(\ref{eq:3.17chen}), we obtain
\begin{align}
    \label{eq:3.21chenu}
    [m(\|u_n\|_W^p)-m(\|u\|_W^p)] \int_{\mathbb{R}^N}\frac{V(x)|u|^{p-2}u(u_n-u)}{|x|^{\alpha}}\dd x \to0,\quad \mbox{as}\ n\to\infty.
\end{align}
From (\ref{eq:3.18chen})-(\ref{eq:3.21chenu}) and \eqref{hip:m1}, we obtain
\begin{align*}
\lim\limits_{n\to\infty}m(\|u_n\|_W^p)\left(\Big[B_{u_n}(u_n-u)- B_{u}(u_n-u)\Big]+\int_{\mathbb{R}^N}\frac{V(x)(|u_n|^{p-2}u_n-|u|^{p-2}u)(u_n-u)}{|x|^{\alpha}}\dd x\right)=0.
\end{align*}
Since $m(\|u_n\|_W^p)[ B_{u_n}(u_n-u)- B_{u}(u_n-u)] \geqslant0$ and $\frac{V(x)(|u_n|^{p-2}u_n-|u|^{p-2}u)(u_n-u)}{|x|^{\alpha}}\geqslant0$ for all $n$ by convexity, \eqref{hip:m1} and $(V_1)$,
we have
\begin{align}\label{eq:3.22chena}
\lim\limits_{n\to\infty} \Big[B_{u_n}(u_n-u)- B_{u}(u_n-u)\Big]=0
\end{align}
and    
\begin{align}\label{eq:3.22chenb}
\lim\limits_{n\to\infty} \int_{\mathbb{R}^N}\frac{V(x)(|u_n|^{p-2}u_n-|u|^{p-2}u)(u_n-u)}{|x|^{\alpha}}\dd x=0.
\end{align}
Let us now recall the well-known Simon inequalities. There
exist positive numbers $c_p$ and $C_p$, depending only
on $p$, such that
\begin{equation} 
\label{simon}
|\xi-\eta|^{p}
\leqslant  \begin{cases}
 c_p (|\xi|^{p-2}\xi-|\eta|^{p-2}\eta)(\xi-\eta)
& \text{for } p\geqslant2,\\[3pt]
C_p\big[(|\xi|^{p-2}\xi-|\eta|^{p-2}\eta)  (\xi-\eta) \big]^{p/2}(|\xi|^p+|\eta|^p)^{(2-p)/2}
& \text{for }1<p<2,
\end{cases}
\end{equation}
for all $\xi,\eta\in\mathbb{R}^N$. According to the Simon
inequality, we divide the discussion into two cases.

\medskip

{\item \scshape Case $p\geqslant 2$:} From (\ref{eq:3.22chena}) and (\ref{simon}), as $n\to\infty$, we have
\begin{align*}
[u_n-u]_{W_{\theta}^{s,p}(\mathbb{R}^N)}^p
& = \iint_{\mathbb{R}^{2N}}\frac{|u_n(x)-u(x)-u_n(y)+u(y)|^p}{|x|^{\theta_1}|x-y|^{N+ps}|y|^{\theta_2}}\dd x\dd y\\
& =  \iint_{\mathbb{R}^{2N}}\frac{|(u_n(x)-u_n(y))-(u(x)-u(y))|^p}{|x|^{\theta_1}|x-y|^{N+ps}|y|^{\theta_2}}\dd x\dd y\\
& \leqslant c_p\iint_{\mathbb{R}^{2N}}\frac{|u_n(x) -u_n(y)|^{p-2}(u_n(x) -u_n(y))-|u(x) -u(y)|^{p-2}(u(x) -u(y))}{|x|^{\theta_1}|x-y|^{N+ps}|y|^{\theta_2}} \nonumber\\
&\quad \times \Big(u_n(x)-u(x)-u_n(y)+u(y)\Big)\dd x\dd y\\
& = c_p\Big[B_{u_n}(u_n-u)- B_{u}(u_n-u)\Big]=o(1),
\end{align*}
and
\begin{align*}
\|u_n-u\|_{L^p_V(\mathbb{R}^N, |x|^{-\alpha})}^p & = \int_{\mathbb{R}^N}\frac{V(x)|u_n-u|^{p}}{|x|^{\alpha}}\dd x \\
& \leqslant c_p\int_{\mathbb{R}^N}\frac{V(x)(|u_n|^{p-2}u_n-|u|^{p-2}u)(u_n-u)}{|x|^{\alpha}}\dd x=o(1).
\end{align*}
Consequently, $\|u_n-u\|_W= ([u_n-u]_{W_{\theta}^{s,p}(\mathbb{R}^N)}^p +\|u_n-u\|_{L^p_V(\mathbb{R}^N, |x|^{-\alpha})}^p)^{\frac{1}{p}} \to0$ as $n\to\infty$.

\medskip 

\item {\scshape Case $1<p<2$:} Taking $\xi=u_n(x)-u_n(y)$ and $\eta=u(x)-u(y)$ in  (\ref{simon}), as $n\to\infty$, we have
\begin{align*}
\lefteqn{[u_n-u]_{W_{\theta}^{s,p}(\mathbb{R}^N)}^p}\\
& =\iint_{\mathbb{R}^{2N}}\frac{|u_n(x)-u(x)-u_n(y)+u(y)|^p}{|x|^{\theta_1}|x-y|^{N+ps}|y|^{\theta_2}}\dd x\dd y\\
& = \iint_{\mathbb{R}^{2N}}\frac{|(u_n(x)-u_n(y))-(u(x)-u(y))|^p}{|x|^{\theta_1}|x-y|^{N+ps}|y|^{\theta_2}}\dd x\dd y\\
& \leqslant  C_p \Bigl[\iint_{\mathbb{R}^{2N}}\frac{|u_n(x) -u_n(y)|^{p-2}(u_n(x) -u_n(y))-|u(x) -u(y)|^{p-2}(u(x) -u(y))}{|x|^{\theta_1}|x-y|^{N+ps}|y|^{\theta_2}} \nonumber\\
& \quad \times \Big(u_n(x)-u(x)-u_n(y)+u(y)\Big)\Bigr]^{\frac{p}{2}}\Bigl(\frac{|u_n(x)-u_n(y)|^p + |u(x)-u(y)|^p)}{{|x|^{\theta_1}|x-y|^{N+ps}|y|^{\theta_2}} }\Bigr)^{\frac{2-p}{2}}\dd x\dd y\\
& = C_p\big[B_{u_n}(u_n-u)-B_{u}(u_n-u) \big]^{p/2}([u_n]_{W^{s,p}_{\theta}(\mathbb{R}^N)}^p+[u]_{W^{s,p}_{\theta}(\mathbb{R}^N)}^p)^{(2-p)/2}\\
& \leqslant C_p\big[B_{u_n}(u_n-u)-B_{u}(u_n-u) \big]^{p/2}([u_n]_{W^{s,p}_{\theta}(\mathbb{R}^N)}^{p(2-p)/2}+[u]_{W^{s,p}_{\theta}(\mathbb{R}^N)}^{p(2-p)/2})\\
& \leqslant C\big[B_{u_n}(u_n-u)-B_{u}(u_n-u) \big]^{p/2}=o(1).
\end{align*}
Here we used the fact that $[u_n]_{W^{s,p}_{\theta}(\mathbb{R}^N)}$ and $[u]_{W^{s,p}_{\theta}(\mathbb{R}^N)}$ are bounded, and the elementary inequality
$$
(a+b)^{(2-p)/2}\leqslant a^{(2-p)/2}+b^{(2-p)/2}\ \ \mbox{for\ all}\ a,b\geqslant 0\ \mbox{and}\ 1<p<2.
$$
Moreover, by H\"{o}lder inequality and (\ref{eq:3.22chenb}), as $n\to\infty$,
\begin{align*}
\lefteqn{\|u_n-u\|_{L^p_V(\mathbb{R}^N, |x|^{-\alpha})}^p} \\
& = \int_{\mathbb{R}^N}\frac{V(x)|u_n-u|^{p}}{|x|^{\alpha}}\dd x \\
&\leqslant C_p\int_{\mathbb{R}^N}V(x) \bigl[\frac{(|u_n|^{p-2}u_n-|u|^{p-2}u)  (u_n-u)}{|x|^{\alpha}} \bigr]^{p/2}\Bigl(\frac{|u_n|^p+|u|^p}{|x|^{\alpha}}\Bigr)^{(2-p)/2}\dd x\\
&\leqslant C_p\left(\int_{\mathbb{R}^N}\frac{V(x)  (|u_n|^{p-2}u_n-|u|^{p-2}u)  (u_n-u)}{|x|^{\alpha}} \dd x\right)^{p/2}\\
& \quad \times\left(\int_{\mathbb{R}^N}\frac{V(x)(|u_n|^p+|u|^p)}{|x|^{\alpha}}\dd x\right)^{(2-p)/2}\\
&\leqslant C_p\left(\|u_n\|_{L^p_V(\mathbb{R}^N, |x|^{-\alpha})}^{p(2-p)/2}+\|u\|_{L^p_V(\mathbb{R}^N, |x|^{-\alpha})}^{p(2-p)/2}\right)\\
& \quad \times \left(\int_{\mathbb{R}^N}\frac{V(x)  (|u_n|^{p-2}u_n-|u|^{p-2}u)  (u_n-u)}{|x|^{\alpha}} \dd x\right)^{p/2}\\
& \leqslant C \left(\int_{\mathbb{R}^N}\frac{V(x)  (|u_n|^{p-2}u_n-|u|^{p-2}u)  (u_n-u)}{|x|^{\alpha}} \dd x\right)^{p/2}\to0.
\end{align*}
Thus $\|u_n-u\|_W=([u_n-u]_{W_{\theta}^{s,p}(\mathbb{R}^N)}^p +\|u_n-u\|_{L^p_V(\mathbb{R}^N, |x|^{-\alpha})}^p)^{\frac{1}{p}}\to0$ as $n\to\infty$. The proof is complete.
\end{proof}

\begin{proof}[Proof of Theorem~\ref{teo:chen}]
By Lemma~\ref{mage} we show the geometry of the functional $I$ associated with the problem~\eqref{problema:Kirchhof}. Furthermore, by Lemmas~\ref{lema:3.4chen} and~\ref{mage2} we show that the Cerami sequences for the functional are limited and that the functional $I$ verifies the Cerami condition, respectively. Therefore, we obtain that there exists a critical point of functional $I$, so problem~\eqref{problema:Kirchhof} has a nontrivial weak solution for any $\lambda>0$.  
\end{proof}

\subsubsection*{Conflict of interest}
On behalf of all authors, the corresponding author states that there is no
conflict of interest.

\subsubsection*{Data availability statement}
Data sharing is not applicable to this article as no new data were created or analyzed in this study.

\subsubsection*{Orcid}

\begin{tabbing}
\hspace{4.5cm}\=\kill
Ronaldo B. Assun\c{c}\~{a}o 
\> \url{https://orcid.org/0000-0003-3159-6815}\\ 
Olímpio H. Miyagaki
\> \url{https://orcid.org/0000-0002-5608-3760} \\ 
Rafaella F. S. Siqueira
\> \url{https://orcid.org/0009-0007-2271-7327}
\end{tabbing}

\bibliographystyle{IEEEtranS}
\bibliography{bibliografia}

% Generated by IEEEtranS.bst, version: 1.14 (2015/08/26)
\begin{thebibliography}{10}
\providecommand{\url}[1]{#1}
\csname url@samestyle\endcsname
\providecommand{\newblock}{\relax}
\providecommand{\bibinfo}[2]{#2}
\providecommand{\BIBentrySTDinterwordspacing}{\spaceskip=0pt\relax}
\providecommand{\BIBentryALTinterwordstretchfactor}{4}
\providecommand{\BIBentryALTinterwordspacing}{\spaceskip=\fontdimen2\font plus
\BIBentryALTinterwordstretchfactor\fontdimen3\font minus \fontdimen4\font\relax}
\providecommand{\BIBforeignlanguage}[2]{{%
\expandafter\ifx\csname l@#1\endcsname\relax
\typeout{** WARNING: IEEEtranS.bst: No hyphenation pattern has been}%
\typeout{** loaded for the language `#1'. Using the pattern for}%
\typeout{** the default language instead.}%
\else
\language=\csname l@#1\endcsname
\fi
#2}}
\providecommand{\BIBdecl}{\relax}
\BIBdecl

\bibitem{agmon1959estimates}
S.~Agmon, A.~Douglis, and L.~Nirenberg, ``Estimates near the boundary for solutions of elliptic partial differential equations satisfying general boundary conditions. i,'' \emph{Communications on pure and applied mathematics}, vol.~12, no.~4, pp. 623--727, 1959.

\bibitem{ahmad1976elementary}
S.~Ahmad, A.~Lazer, and J.~Paul, ``Elementary critical point theory and perturbations of elliptic boundary value problems at resonance,'' \emph{Indiana University Mathematics Journal}, vol.~25, no.~10, pp. 933--944, 1976.

\bibitem{alves1996radially}
C.~O. Alves, D.~C. de~Morais~Filho, and M.~A.~S. Souto, ``Radially symmetric solutions for a class of critical exponent elliptic problems in $\mathbb{R}^n$,'' \emph{Electronic Journal of Differential Equation}, 1996.

\bibitem{alves2012existence}
C.~O. Alves and M.~A. Souto, ``Existence of solutions for a class of elliptic equations in $\mathbb{R}^{N}$ with vanishing potentials,'' \emph{Journal of Differential Equations}, vol. 252, no.~10, pp. 5555--5568, 2012.

\bibitem{alves2013existence}
------, ``Existence of solutions for a class of nonlinear {S}chr{\"o}dinger equations with potential vanishing at infinity,'' \emph{Journal of Differential Equations}, vol. 254, no.~4, pp. 1977--1991, 2013.

\bibitem{alves2010class}
C.~O. Alves, F.~Corr\^{e}a, and G.~M. Figueiredo, ``On a class of nonlocal elliptic problems with critical growth,'' \emph{Differ. Equ. Appl}, vol.~2, no.~3, pp. 409--417, 2010.

\bibitem{alves2015existence}
M.~J. Alves, R.~B. Assun{\c{c}}{\~{a}}o, and O.~H. Miyagaki, ``Existence result for a class of quasilinear elliptic equations with ($ p $-$ q $)-laplacian and vanishing potentials,'' \emph{Illinois Journal of Mathematics}, vol.~59, no.~3, pp. 545--575, 2015.

\bibitem{alves2022existence}
M.~J. Alves and R.~B. Assun\c{c}\~{a}o, ``Existence of solutions for a problem with multiple singular weighted p-laplacians and vanishing potentials,'' \emph{Electronic Journal of Differential Equations}, 2022.

\bibitem{amann1980nontrivial}
H.~Amann and E.~Zehnder, ``Nontrivial solutions for a class of nonresonance problems and applications to nonlinear differential equations,'' \emph{Annali della Scuola Normale Superiore di Pisa-Classe di Scienze}, vol.~7, no.~4, pp. 539--603, 1980.

\bibitem{assunção2023fractional}
R.~B. Assunção, O.~H. Miyagaki, and R.~F.~S. Siqueira, ``Fractional {S}obolev-{C}hocard critical equation with {H}ardy term and weighted singularities,'' 2023, \url{https://arxiv.org/abs/2311.00852}.

\bibitem{autuori2015stationary}
G.~Autuori, A.~Fiscella, and P.~Pucci, ``Stationary {K}irchhoff problems involving a fractional elliptic operator and a critical nonlinearity,'' \emph{Nonlinear Analysis}, vol. 125, pp. 699--714, 2015.

\bibitem{bartolo1983abstract}
P.~Bartolo, V.~Benci, and D.~Fortunato, ``Abstract critical point theorems and applications to some nonlinear problems with “strong” resonance at infinity,'' \emph{Nonlinear analysis: Theory, methods \& applications}, vol.~7, no.~9, pp. 981--1012, 1983.

\bibitem{bartsch1995existence}
T.~Bartsch and Z.~Q. Wang, ``Existence and multiplicity results for some superlinear elliptic problems on $\mathbb{R}^n$: existence and multiplicity results,'' \emph{Communications in Partial Differential Equations}, vol.~20, no. 9-10, pp. 1725--1741, 1995.

\bibitem{berestycki1983nonlinear}
H.~Berestycki and P.-L. Lions, ``Nonlinear scalar field equations. pt. 1,'' \emph{Archive for Rational Mechanics and Analysis}, vol.~82, no.~4, pp. 313--346, 1983.

\bibitem{brezis2011functional}
H.~Br\'{e}zis, \emph{Functional analysis, {S}obolev spaces and partial differential equations}.\hskip 1em plus 0.5em minus 0.4em\relax Springer, 2011, vol.~2.

\bibitem{castro1979critical}
\BIBentryALTinterwordspacing
A.~Castro and A.~C. Lazer, ``Critical point theory and the number of solutions of a nonlinear {D}irichlet problem,'' \emph{Ann. Mat. Pura Appl. (4)}, vol. 120, pp. 113--137, 1979. [Online]. Available: \url{https://doi.org/10.1007/BF02411940}
\BIBentrySTDinterwordspacing

\bibitem{chen2013multiple}
S.-J. Chen and L.~Li, ``Multiple solutions for the nonhomogeneous {K}irchhoff equation on $\mathbb{R}^{N}$,'' \emph{Nonlinear Analysis: Real World Applications}, vol.~14, no.~3, pp. 1477--1486, 2013.

\bibitem{chen2019existence}
\BIBentryALTinterwordspacing
W.~Chen, ``Existence of solutions for fractional {$p$}-{K}irchhoff type equations with a generalized {C}hoquard nonlinearity,'' \emph{Topol. Methods Nonlinear Anal.}, vol.~54, no.~2, pp. 773--791, 2019. [Online]. Available: \url{https://doi.org/10.12775/tmna.2019.069}
\BIBentrySTDinterwordspacing

\bibitem{clement}
P.~Cl\'{e}ment, D.~de~Figueiredo, and E.~Mitidieri, ``Positive solutions of semilinear elliptic systems,'' \emph{Communications in Partial Differential Equations}, vol.~17, pp. 923--940, 01 1992.

\bibitem{colasuonno2011multiplicity}
F.~Colasuonno and P.~Pucci, ``Multiplicity of solutions for $p (x)$-polyharmonic elliptic {K}irchhoff equations,'' \emph{Nonlinear Analysis: Theory, Methods \& Applications}, vol.~74, no.~17, pp. 5962--5974, 2011.

\bibitem{costa1994class}
D.~G. Costa, ``On a class of elliptic systems in {$\mathbb{R}^N$},'' \emph{Electron. J. Differential Equations}, pp. No. 07, approx. 14 pp., 1994.

\bibitem{costa2023remarks}
------, ``Remarks on compactness conditions and their application,'' \emph{Electronic Journal of Differential Equations}, vol. Special Issue 02, pp. 101--107, 2023.

\bibitem{costa1995nontrivial}
D.~G. Costa and O.~H. Miyagaki, ``Nontrivial solutions for perturbations of the $p$-laplacian on unbounded domains,'' \emph{Journal of Mathematical Analysis and Applications}, vol. 193, no.~3, pp. 737--755, 1995.

\bibitem{sissa1992homoclinic}
V.~Coti~Zelati and P.~H. Rabinowitz, ``Homoclinic type solutions for a semilinear elliptic pde on $\mathbb{R}^{N}$,'' \emph{Communications on pure and applied mathematics}, vol.~45, no.~10, pp. 1217--1269, 1992.

\bibitem{de2003semilinear}
D.~G. de~Figueiredo and J.~Yang, ``On a semilinear elliptic problem without ({PS}) condition,'' \emph{Journal of Differential Equations}, vol. 187, no.~2, pp. 412--428, 2003.

\bibitem{MR2944369}
\BIBentryALTinterwordspacing
E.~Di~Nezza, G.~Palatucci, and E.~Valdinoci, ``Hitchhiker's guide to the fractional {S}obolev spaces,'' \emph{Bull. Sci. Math.}, vol. 136, no.~5, pp. 521--573, 2012. [Online]. Available: \url{https://doi.org/10.1016/j.bulsci.2011.12.004}
\BIBentrySTDinterwordspacing

\bibitem{figueiredo2013existence}
G.~M. Figueiredo, ``Existence of a positive solution for a {K}irchhoff problem type with critical growth via truncation argument,'' \emph{Journal of Mathematical Analysis and Applications}, vol. 401, no.~2, pp. 706--713, 2013.

\bibitem{fiscella2014critical}
A.~Fiscella and E.~Valdinoci, ``A critical {K}irchhoff type problem involving a nonlocal operator,'' \emph{Nonlinear Analysis: Theory, Methods \& Applications}, vol.~94, pp. 156--170, 2014.

\bibitem{harrabi2014palais}
A.~Harrabi, ``On the {P}alais-{S}male condition,'' \emph{Journal of Functional Analysis}, vol. 267, no.~8, pp. 2995--3015, 2014.

\bibitem{jeanjean1999existence}
L.~Jeanjean, ``On the existence of bounded {P}alais-{S}male sequences and application to a landesman--lazer-type problem set on $\mathbb{R}^n$,'' \emph{Proceedings of the Royal Society of Edinburgh Section A: Mathematics}, vol. 129, no.~4, pp. 787--809, 1999.

\bibitem{jianfu1987existence}
Y.~Jianfu and Z.~Xiping, ``On the existence of nontrivial solution of a quasilinear elliptic boundary value problem for unbounded domains:(i) positive mass case,'' \emph{Acta Mathematica Scientia}, vol.~7, no.~3, pp. 341--359, 1987.

\bibitem{kryszewski1998generalized}
W.~Kryszewski and A.~Szulkin, ``Generalized linking theorem with an application to a semilinear {S}chr\"{o}dinger equation,'' \emph{Adv. Differential Equations}, vol.~3, no.~3, pp. 441--472, 1998.

\bibitem{landesman1970nonlinear}
E.~M. Landesman and A.~Lazer, ``Nonlinear perturbations of linear elliptic boundary value problems at resonance,'' \emph{Journal of Mathematics and Mechanics}, vol.~19, no.~7, pp. 609--623, 1970.

\bibitem{liang2013soliton}
S.~Liang and S.~Shi, ``Soliton solutions to {K}irchhoff type problems involving the critical growth in $\mathbb{R}^{N}$,'' \emph{Nonlinear Analysis: Theory, Methods \& Applications}, vol.~81, pp. 31--41, 2013.

\bibitem{MR717827}
\BIBentryALTinterwordspacing
E.~H. Lieb, ``Sharp constants in the {H}ardy-{L}ittlewood-{S}obolev and related inequalities,'' \emph{Ann. of Math. (2)}, vol. 118, no.~2, pp. 349--374, 1983. [Online]. Available: \url{https://doi.org/10.2307/2007032}
\BIBentrySTDinterwordspacing

\bibitem{MR1817225}
\BIBentryALTinterwordspacing
E.~H. Lieb and M.~Loss, \emph{Analysis}, 2nd~ed., ser. Graduate Studies in Mathematics.\hskip 1em plus 0.5em minus 0.4em\relax American Mathematical Society, Providence, RI, 2001, vol.~14. [Online]. Available: \url{https://doi.org/10.1090/gsm/014}
\BIBentrySTDinterwordspacing

\bibitem{miyagaki1997class}
O.~H. Miyagaki, ``On a class of semilinear elliptic problems in $\mathbb{R}^n$ with critical growth,'' \emph{Nonlinear Analysis: Theory, Methods \& Applications}, vol.~29, no.~7, pp. 773--781, 1997.

\bibitem{MR3445279}
\BIBentryALTinterwordspacing
G.~Molica~Bisci, V.~R\u{a}dulescu, and R.~Servadei, \emph{Variational methods for nonlocal fractional problems}, ser. Encyclopedia of Mathematics and its Applications.\hskip 1em plus 0.5em minus 0.4em\relax Cambridge University Press, Cambridge, 2016, vol. 162, with a foreword by Jean Mawhin. [Online]. Available: \url{https://doi.org/10.1017/CBO9781316282397}
\BIBentrySTDinterwordspacing

\bibitem{MR3625092}
\BIBentryALTinterwordspacing
V.~Moroz and J.~Van~Schaftingen, ``A guide to the {C}hoquard equation,'' \emph{J. Fixed Point Theory Appl.}, vol.~19, no.~1, pp. 773--813, 2017. [Online]. Available: \url{https://doi.org/10.1007/s11784-016-0373-1}
\BIBentrySTDinterwordspacing

\bibitem{pankov2004periodic}
A.~Pankov, ``Periodic nonlinear {S}chr\"odinger equation with application to photonic crystals,'' 2004.

\bibitem{MR3975603}
\BIBentryALTinterwordspacing
P.~Pucci, M.~Xiang, and B.~Zhang, ``Existence results for {S}chr\"{o}dinger-{C}hoquard-{K}irchhoff equations involving the fractional {$p$}-{L}aplacian,'' \emph{Adv. Calc. Var.}, vol.~12, no.~3, pp. 253--275, 2019. [Online]. Available: \url{https://doi.org/10.1515/acv-2016-0049}
\BIBentrySTDinterwordspacing

\bibitem{ramos2001fourth}
M.~Ramos and P.~Rodrigues, ``On a fourth order superlinear elliptic problem,'' \emph{Electronic Journal of Differential Equations}, pp. 243--255, 2001.

\bibitem{robinson1995multiple}
S.~B. Robinson, ``Multiple solutions for semilinear elliptic boundary value problems at resonance,'' \emph{Electron. J. Differential Equations}, pp. Paper No. 1, 14, 1995.

\bibitem{Schechter}
\BIBentryALTinterwordspacing
M.~Schechter, ``The use of {C}erami sequences in critical point theory,'' \emph{Abstr. Appl. Anal.}, pp. Art. ID 58\,948, 28, 2007. [Online]. Available: \url{https://doi.org/10.1155/2007/58948}
\BIBentrySTDinterwordspacing

\bibitem{song2017existence}
Y.~Song and S.~Shi, ``Existence of infinitely many solutions for degenerate p-fractional {K}irchhoff equations with critical {S}obolev--{H}ardy nonlinearities,'' \emph{Zeitschrift f{\"u}r angewandte Mathematik und Physik}, vol.~68, pp. 1--13, 2017.

\bibitem{MR0098285}
\BIBentryALTinterwordspacing
E.~M. Stein and G.~Weiss, ``Fractional integrals on {$n$}-dimensional {E}uclidean space,'' \emph{J. Math. Mech.}, vol.~7, pp. 503--514, 1958. [Online]. Available: \url{https://doi.org/10.1512/iumj.1958.7.57030}
\BIBentrySTDinterwordspacing

\bibitem{thews1980non}
K.~Thews, ``Non-trivial solutions of elliptic equations at resonance,'' \emph{Proceedings of the Royal Society of Edinburgh Section A: Mathematics}, vol.~85, no. 1-2, pp. 119--129, 1980.

\end{thebibliography}

\end{document}